\documentclass[11pt]{article}
\usepackage[dvips]{graphicx}
\usepackage[dvips,letterpaper,margin=1.5in]{geometry}
\usepackage{color}
\usepackage{amsmath}
\usepackage{amsfonts}
\usepackage{amssymb}
\usepackage{amsthm}
\usepackage{newlfont}
\usepackage{epsfig}
\usepackage{multirow}
\usepackage{mhchem}
\usepackage{enumerate}
\usepackage{setspace}
\usepackage{cite}
\usepackage{tikz}
\usepackage{subfig}
\usetikzlibrary{shapes,arrows,automata}
\usepackage{colortbl}
\usepackage{algorithm}
\usepackage{algpseudocode}
\usepackage{multicol}
\usetikzlibrary{shapes,fit}
\usetikzlibrary{calc}
\usepackage{appendix}

\begin{document}

\newtheorem{assumption}{Assumption}[section]
\newtheorem{definition}{Definition}[section]
\newtheorem{lemma}{Lemma}[section]
\newtheorem{proposition}{Proposition}[section]
\newtheorem{theorem}{Theorem}[section]
\newtheorem{corollary}{Corollary}[section]
\newtheorem{remark}{Remark}[section]
\newtheorem{conjecture}{Conjecture}[section]
\newtheorem{example}{Example}[section]

\small

\title{A Computational Approach to Extinction Events in Chemical Reaction Networks with Discrete State Spaces}
\author{Matthew D. Johnston\\Department of Mathematics\\San Jos\'{e} State University\\San Jos\'{e}, CA, 95192 USA}
\date{}
\maketitle


\vspace{0.2in}

\begin{abstract}
\small

Recent work of M.D. Johnston \emph{et al.} has produced sufficient conditions on the structure of a chemical reaction network which guarantee that the corresponding discrete state space system exhibits an extinction event. The conditions consist of a series of systems of equalities and inequalities on the edges of a modified reaction network called a domination-expanded reaction network. In this paper, we present a computational implementation of these conditions written in Python and apply the program on examples drawn from the biochemical literature, including a model of polyamine metabolism in mammals and a model of the pentose phosphate pathway in \emph{Trypanosoma brucei}. We also run the program on 458 models from the European Bioinformatics Institute's BioModels Database and report our results.
\end{abstract}

\noindent \textbf{Keywords:} reaction network, discrete extinction, stochastic process, algorithms, linear programming \newline \textbf{AMS Subject Classifications:} 	92C42, 90C90

\bigskip

\section{Introduction}
\label{introduction}

A chemical reaction network describes the conversion of chemical reactants into products as a directed graph where the vertices are aggregates of species called complexes and the edges are reactions. A dynamical system can be associated to such a network in several ways. When the counts of reacting species are high, as is typical in industrial chemistry and pharmacology, the system can be reasonably modeled with a system of differential equations over the continuous state space of reactant concentrations. When the count of reacting species is low, however, as is typical in genetic and enzymatic systems in systems biology, it is more accurate to model the system over the discrete state space of molecular counts, for instance as a continuous-time Markov chain \cite{Kurtz2,AndKurtz2015,AndKurtz2011} or stochastic Petri Net \cite{Bause-Kritzinger,PetriNet/review/Pet1977}.

Recent work has focused on when the long-term behaviors predicted by the continuous and discrete state models are markedly different. While classical work of T. Kurtz guarantees that the probability density function of the discrete model converges in an appropriate scaling limit to the solution of the continuous state model on compact time intervals $[0,T]$, $T < \infty$ \cite{Kurtz2}, on the unbounded interval $[0,\infty)$ this convergence may fail. For example, consider the following network, which was presented as Example 2.1 in the companion paper \cite{J-A-C-B}:
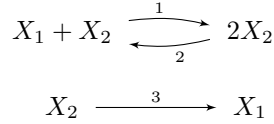
\begin{center}
\begin{tikzpicture}[auto, outer sep=3pt, node distance=2cm,>=latex']
\node (C1) {$X_1+X_2$};
\node [right of = C1, node distance = 2.5cm] (C2) {$2X_2$};
\node [below of = C2, node distance = 1.0cm] (C3) {$X_1$};
\node [left of = C3, node distance = 2.5cm] (C4) {$X_2$};
\path[->, bend left = 10] (C1) edge node [above left = -0.15cm] {\tiny $1$} (C2);
\path[->, bend left = 10] (C2) edge node [below right = -0.15cm] {\tiny $2$} (C1);
\path[->] (C4) edge node [above = -0.15cm] {\tiny $3$} (C3);
\end{tikzpicture}
\end{center}
where the numbers correspond to the enumeration of the reactions. For almost all initial conditions and parameter values, the continuous state differential equation model predicts convergence to a strictly positive steady state. For the discrete state model, however, the reaction $X_2 \to X_1$ may irreversibly deplete the species $X_2$ so that the inevitable final state of the system for $M = \mathbf{X}_1(0) + \mathbf{X}_2(0)$ is $\{ \mathbf{X}_1 = M, \mathbf{X}_2 = 0\}$.

Extinction events in discrete state space models were studied in the context of chemical reaction network theory by D.F. Anderson \emph{et al.} in \cite{A-E-J}. In that paper, the authors showed that a large subset of networks which exhibit ``absolute concentration robustness'' in the continuous state model \cite{Sh-F} exhibit an extinction event when modeled with a discrete state space. A primary extension made in that paper was generalizing the notion of ``differing in one species'' to a domination relationship between complexes. This domination relation, and the corresponding relationship with extinction events in discrete state space systems, were further generalized and adopted to the conventions of Petri Net Theory by R. Brijder in \cite{Brijder}. Finally, in this paper's companion paper \cite{J-A-C-B}, M.D. Johnston \emph{et al.} clarify this relationship by presenting sufficient conditions for an extinction event which depend upon the evaluation of computationally-tractable systems of equalities and inequalities.

The method presented in \cite{J-A-C-B} involves creating a modified reaction network called a domination-expanded network. For example, for the network above, we may correspond the following expanded network:
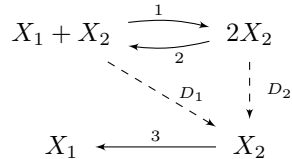
\begin{center}
\begin{tikzpicture}[auto, outer sep=3pt, node distance=2cm,>=latex']
\node (C1) {$X_1+X_2$};
\node [right of = C1, node distance = 2.5cm] (C2) {$2X_2$};
\node [below of = C2, node distance = 1.5cm] (C3) {$X_2$};
\node [left of = C3, node distance = 2.5cm] (C4) {$X_1$};
\path[->, bend left = 10] (C1) edge node [above left = -0.15cm] {\tiny $1$} (C2);
\path[->, bend left = 10] (C2) edge node [below right = -0.15cm] {\tiny $2$} (C1);
\path[->] (C3) edge node [above = -0.15cm] {\tiny $3$} (C4);
\path[dashed,->] (C1) edge [right] node {\tiny $D_1$} (C3);
\path[dashed,->] (C2) edge [right] node {\tiny $D_2$} (C3);
\end{tikzpicture}
\end{center}
where the dashed edges are called ``domination reactions.'' This network is then associated by certain rules with some systems of equalities and inequalities. The main result of \cite{J-A-C-B} states that, \emph{if any of these systems cannot be satisfied then the discrete state space system must exhibit an extinction event}.

For example, to the domination-expanded network above we associate the following system of equalities and inequalities on the vector $\alpha = ((\alpha_R)_1,(\alpha_R)_2,(\alpha_R)_3,(\alpha_D)_1,(\alpha_D)_2)) \in \mathbb{Z}_{\geq 0}^5$, where the entries correspond to the reaction edges:
\[\left\{ \; \; \begin{split} & \hspace{0.17in} (\alpha_R)_1 = 0, \; (\alpha_D)_2 = 0 \\ & -(\alpha_R)_1 + (\alpha_R)_2 + (\alpha_R)_3 = 0 \\ & \hspace{0.17in} (\alpha_R)_1 - (\alpha_R)_2 - (\alpha_R)_3 = 0 \\ & \hspace{0.17in} (\alpha_R)_3 \geq (\alpha_D)_1 \geq (\alpha_R)_2 \geq 0. \end{split} \right.\]
Since this system has no nontrivial solution, we may conclude that the discrete state space system has an extinction event.

In this paper, we implement the conditions of the companion paper \cite{J-A-C-B} for affirming an extinction event in discrete state space systems into a computational package written in Python. The program utilizes a series of mixed-integer linear programming (MILP) modules for verifying the technical conditions required to generate the governing systems of equalities and inequalities, and also for evaluating these resulting systems. We then run the algorithm on 458 models from the European Bioinformatics Institute's BioModels Database and report our results. As an illustration of the power of the program, we further analyze a model of polyamine metabolism in mammals and a model of the pentose phosphate pathway in \emph{Trypanosoma brucei}, which were identified by the BioModels database run as exhibiting an extinction event.

We adopt the following notation throughout the paper:
\begin{itemize}
\item
$\mathbb{R}_{\geq 0} = \{x \in \mathbb{R} \mid x \geq 0\}$ and $\mathbb{R}_{> 0} = \{x \in \mathbb{R} \mid x > 0\}$,
\item
for $\mathbf{v} = (v_1,\ldots,v_n) \in \mathbb{R}_{\geq 0}^n$, we define $\mathrm{supp}(\mathbf{v}) = \{ i \in \{1,\dots,n\} \; | \; v_i > 0 \}$,
\item
for a set $X = \{ X_1, X_2, \ldots, X_n \}$ of indexed elements and a subset $W \subseteq X$, we define $\mathrm{supp}(W) = \{ i \in \{1,\dots,n\} \; | \; X_i \in W \}$,
\item
for a subset $W \subseteq X$, we define the complement $W^c = \{ x \in X \; | \; x \not\in W \}$,
\item
for $\mathbf{v}, \mathbf{w} \in \mathbb{R}^n$, we define $\mathbf{v} \leq \mathbf{w}$ if $v_i \leq w_i$ for each $i \in \{1,\dots,n\}$.
\end{itemize}

\section{Background}
\label{background}

In this section, we outlined the definitions, terminology, and background results necessary for our study of extinction events. We follow the notation of \emph{chemical reaction network theory} (CRNT) \cite{F3}. As this background largely follows that of the companion paper \cite{J-A-C-B}, we will be brief and use the example presented in the introduction as a running example.


\subsection{Chemical Reaction Networks}
\label{crntsection}


The following is the basic object of study in CRNT.

\begin{definition}
\label{crn}
A \textbf{chemical reaction network} (CRN) is given by a triple of finite sets $(\mathcal{S},\mathcal{C},\mathcal{R})$ where:
\begin{enumerate}
\item
The \textbf{species set} $\mathcal{S} = \{ X_1, \ldots, X_m \}$ contains the species of the network.
\item
The \textbf{complex set} $\mathcal{C} = \{ y_1, \ldots, y_n \}$ contains linear combination of the species of the form $y_i = \sum_{j=1}^m y_{ij} X_j$. The values $y_{ij} \in \mathbb{Z}_{\geq 0}$ are called \textbf{stoichiometric coefficients}. Allowing a slight abuse of notation, we will let $y_i$ denote both the $i$th complex and the corresponding complex vector $y_i = (y_{i1}, y_{i2}, \ldots, y_{im}) \in \mathbb{Z}_{\geq 0}^m$.
\item
The \textbf{reaction set} $\mathcal{R} = \{ R_1, \ldots, R_r \}$ consists of ordered pairs of complexes, e.g. $(y_i,y_j) \in \mathcal{R}$. We furthermore define the mappings $\rho, \rho': \{ 1, \ldots, r \} \mapsto \{ 1, \ldots, n \}$ such that $\rho(k) = i$ if $y_i$ is the source complex (tail of arrow) of the $k$th reaction, and $\rho'(k) = j$ if $y_j$ is the product complex (head of arrow) of the $k$th reaction, so that we may write $R_k = (y_{\rho(k)},y_{\rho'(k)})$. We also represent reactions with directed arrows, e.g. $y_i \to y_j$ or $y_{\rho(k)} \to y_{\rho'(k)}$.
\end{enumerate}
\end{definition}

\noindent We will assume that (i) every species is contained in at least one complex; and (ii) every complex is contained in at least one reaction.



The interpretation of reactions as directed edges naturally gives rise to a reaction graph $G=(V,E)$ where the set of vertices is given by the complexes (i.e.~$V = \mathcal{C}$) and the set of edges is given by the reactions (i.e.~$E = \mathcal{R})$. A maximal set of connected complexes is called a \emph{linkage class} (LC) while the maximal set of strongly connected complexes is called a \emph{strong linkage class} (SLC). An SLC will be called \emph{terminal} if there are no outgoing edges. A complex is terminal if it is contained in a terminal SLC. A subset of complexes $\mathcal{Y} \subseteq \mathcal{C}$ is called an \emph{absorbing complex set} if it contains every terminal complex and has no outgoing edges, i.e. $y_i \to y_j \in \mathcal{R}$ and $y_i \in \mathcal{Y}$ implies $y_j \in \mathcal{Y}$. Complexes $y_i \in \mathcal{Y}$ are called $\mathcal{Y}$-interior while complexes $y_i \not\in \mathcal{Y}$ are called $\mathcal{Y}$-exterior. Similarly, reactions $y_i \to y_j \in \mathcal{R}$ where $y_i \in \mathcal{Y}$ are called $\mathcal{Y}$-interior and $\mathcal{Y}$-exterior otherwise.

%
%
%
%

We associate the following structural matrices to a CRN. These matrices will be useful in both defining the dynamical system we are interested in, and in the computational implementation of our conditions for an extinction event.

\begin{enumerate}
\item
The \textbf{complex matrix} $Y \in \mathbb{Z}_{\geq 0}^{m \times n}$ is the matrix which has columns $Y_{\cdot,i} = y_i$, $i \in \{1, \ldots, n\},$ where $y_i$ is the complex vector for the $i$th complex.
\item
The \textbf{adjacency matrix} $I_a \in \{ -1, 0, 1 \}^{n \times r}$ is the matrix with entries $[I_a]_{i,j} = -1$ if $\rho(j)=i$, $[I_a]_{i,j} = 1$ if $\rho'(j)=i$, and $[I_a]_{i,j} = 0$ otherwise.
\item
The \textbf{stoichiometric matrix} $\Gamma \in \mathbb{Z}^{m \times r}$ is the matrix with columns $\Gamma_{\cdot,k} = y_{\rho'(k)}-y_{\rho(k)}$ for $k=\{ 1, \ldots, r\}$.
\item
The \textbf{source matrix} $I_s \in \{ 0, 1 \}^{r \times n}$ is the matrix with entries $[I_s]_{i,j} = 1$ if $\rho(i)=j$, and $[I_s]_{i,j} = 0$ otherwise.
\item
The \textbf{unweighted Laplacian} $A \in \mathbb{Z}^{n \times n}$ is the matrix with entries $A_{i,j} = -\mbox{outdeg}(y_i)$ if $i=j$, $A_{i,j} = 1$ if $y_j \to y_i \in \mathcal{R}$, and $A_{i,j} = 0$ otherwise.
\end{enumerate}

\noindent Notice that $\Gamma = Y \cdot I_a$ and $A = I_a \cdot I_s$, where $\Gamma$ is the stoichiometric matrix, and that $I_s$ may be obtained by replacing the negative entries in $I_a$ with ones, setting the remaining entries to zero, and taking the transpose. It therefore suffices to know $Y$ (stoichiometry of CRN) and $I_a$ (connectivity of CRN) in order to generate all of the required structural data of a given CRN.

It is typical in CRNT to use a Laplacian where the entries are weighted by the rate constant of the corresponding reaction \cite{H-J1,F3}. By contrast, the unweighted Laplacian simply encodes the connectivity structure of the complexes in the reaction graph. Although we will not need to consider weighted Laplacians in this paper, we will use well-known properties of them to relate $\mbox{ker}(A)$ to the terminal SLCs of the reaction graph \cite{H-F}. 

To each reaction $y_i \to y_j \in \mathcal{R}$ we associate a \emph{reaction vector} $y_j - y_i \in \mathbb{Z}^m$ which tracks the net gain and loss of each chemical species as a result of the  occurrence of this reaction. The \emph{stoichiometric subspace} is defined as
\[S = \mbox{span} \left\{ y_j - y_i \in \mathbb{Z}^m \; | \; y_i \to y_j \in \mathcal{R} \right\}.\]
\noindent A CRN is \emph{conservative} if there exists a $\mathbf{c} \in \mathbb{R}_{> 0}^m$ such that $\mathbf{c}^T \Gamma = \mathbf{0}^T$, while it is \emph{subconservative} if $\mathbf{c}^T \Gamma \leq \mathbf{0}^T$.

\begin{example}
\label{example1}
Reconsider the CRN given in the introduction, which was taken from Example 2.1 of \cite{J-A-C-B}. We have the sets $\mathcal{S} = \{ X_1, X_2\}$, $\mathcal{R} = \{ X_1 + X_2 \to 2X_2, 2X_2 \to X_1 + X_2, X_2 \to X_1 \}$, and $\mathcal{C} = \{ X_1 + X_2, 2X_2, X_2, X_1 \}$. We have the following structural matrices:
\[\begin{split} & Y = \left[ \begin{array}{cccc} 1 & 0 & 0 & 1 \\ 1 & 2 & 1 & 0 \end{array} \right], \; \; \; I_a = \left[ \begin{array}{ccc} -1 & 1 &0 \\ 1 & -1 & 0 \\ 0 & 0 & -1 \\ 0 & 0 & 1 \end{array}\right], \; \; \; \Gamma = Y \cdot I_a = \left[ \begin{array}{ccc} -1 & 1 & -1 \\ 1 & -1 & 1 \end{array} \right] \\& I_s = \left[ \begin{array}{cccc} 1 & 0 & 0 & 0\\ 0 & 1 & 0 & 0 \\ 0 & 0 & 1 & 0 \end{array} \right], \; \; \; A = I_a \cdot I_s = \left[ \begin{array}{cccc} -1 & 1 & 0 & 0 \\ 1 & -1 & 0 & 0 \\ 0 & 0 & -1 & 0 \\ 0 & 0 & 1 & 0 \end{array} \right]\end{split}\]
\noindent The stoichiometric subspace is given by $S = \mbox{span} \{ (1,-1) \}$ and the CRN is conservative with respect to the vector $\mathbf{c} = (1,1)$. This represents the observation that $X_1 + X_2 $ is constant.
\end{example}

\subsection{Chemical Reaction Networks with Discrete State Spaces}
\label{stochasticsection}

The evolution of a CRN on the discrete state space $\mathbb{Z}_{\geq 0}^m$ is given by
\begin{equation}
\label{markov}
\mathbf{X}(t) = \mathbf{X}(0) + \Gamma \; \mathbf{N}(t)
\end{equation}
where $\mathbf{X}(t) = (\mathbf{X}_1(t), \ldots, \mathbf{X}_m(t)) \in \mathbb{Z}_{\geq 0}^m$ is the \emph{discrete state} of the system at time $t$, and $\mathbf{N}(t) = (N_1(t),\ldots,N_r(t))$ and $N_{k}(t) \in \mathbb{Z}_{\geq 0}$ is the number of times the $k$th reaction has occurred up to time $t$. Several frameworks exist for precisely modeling the stochastic evolution of the discrete state $\mathbf{X}(t)$ over time, including the theories of continuous-time Markov chains (CTMC) \cite{Lawler,AndKurtz2015} and stochastic Petri Nets \cite{Bause-Kritzinger,PetriNet/review/Pet1977}. We will not be interested in these details; rather, we will be interested in where in the discrete state space $\mathbb{Z}_{\geq 0}^m$ the trajectory $\mathbf{X}(t)$ may travel. For a similar treatment, see \cite{P-C-K}.

We adopt the following definitions from \cite{J-A-C-B}.
\begin{definition}
\label{recurrence}
Consider a CRN on a discrete state space. Then:
\begin{enumerate}
\item
A complex $y_i \in \mathcal{C}$ is \textbf{charged} at state $\mathbf{X} \in \mathbb{Z}_{\ge 0}$ if $\mathbf{X}_j \geq y_{ij}$ for all $j\in \{1,\dots,m\}$.
\item
A reaction $y_i \to y_j \in \mathcal{R}$ is \textbf{charged} at state $\mathbf{X} \in \mathbb{Z}_{\ge 0}$ if $y_i$ is charged at $\mathbf{X}$.
\item
A state $\mathbf{X} \in \mathbb{Z}_{\geq 0}^m$ \textbf{reacts to} a state $\mathbf{Y} \in \mathbb{Z}_{\geq 0}^m$ (denoted $\mathbf{X} \to \mathbf{Y}$) if there is a reaction $y_i \to y_j \in \mathcal{R}$ such that $\mathbf{Y} = \mathbf{X} + y_j - y_i$ and $y_i$ is charged at state $\mathbf{X}$.
\item
A state $\mathbf{Y} \in \mathbb{Z}_{\geq 0}^m$ is \textbf{reachable} from a state $\mathbf{X} \in \mathbb{Z}_{\geq 0}^m$ (denoted $\mathbf{X} \leadsto \mathbf{Y})$ if there exists a sequence of states such that $\mathbf{X} = \mathbf{X}_{\nu(1)} \to \mathbf{X}_{\nu(2)} \to \cdots \to \mathbf{X}_{\nu(l)} = \mathbf{Y}$.
\item
A state $\mathbf{X} \in \mathbb{Z}_{\geq 0}^m$ is \textbf{recurrent} if, for any $\mathbf{Y} \in \mathbb{Z}_{\geq 0}^m$, $\mathbf{X} \leadsto \mathbf{Y}$ implies $\mathbf{Y} \leadsto \mathbf{X}$; otherwise, the state is \textbf{transient}.
\item
A complex $y_i \in \mathcal{C}$ is \textbf{recurrent} from state $\mathbf{X} \in \mathbb{Z}_{\geq 0}^m$ if $\mathbf{X} \leadsto \mathbf{Y}$ 
implies that there is a $\mathbf{Z}$ for which $\mathbf{Y} \leadsto \mathbf{Z}$
and $y_i$ is charged at $\mathbf{Z}$; otherwise, $y_i$ is \textbf{transient} from $\mathbf{X}$.
\item
A reaction $y_i \to y_j \in \mathcal{R}$ is \textbf{recurrent} from state $\mathbf{X} \in \mathbb{Z}_{\geq 0}^m$ if $y_i$ is recurrent from $\mathbf{X}$; otherwise, $y_i \to y_j \in \mathcal{R}$ is \textbf{transient} from $\mathbf{X}$
\item
The CRN exhibits an \textbf{extinction event on $\mathcal{Y} \subseteq \mathcal{C}$ from $\mathbf{X}\in \mathbb{Z}^m_{\ge 0}$} if every complex $y_i \in \mathcal{Y}$ is transient from $\mathbf{X}$.
\item
The CRN exhibits a \textbf{guaranteed extinction event on $\mathcal{Y} \subseteq \mathcal{C}$} if it has an extinction event on $\mathcal{Y}$ from every $\mathbf{X}\in \mathbb{Z}^m_{\ge 0}$.
\end{enumerate}
\end{definition}

\noindent Since subconservative CRNs on discrete state spaces have a finite number of states (see Theorem 1, \cite{DBLP:conf/ac/MemmiR75}), the notion of recurrence presented above corresponds to the notion of positive recurrence from the theory of stochastic process \cite{Lawler,AndKurtz2015}.

\begin{example}
\label{example2}
Reconsider the CRN introduced in Example \ref{example1} and the initial state $\mathbf{X} = \{ \mathbf{X}_1 = 0, \mathbf{X}_2 = 3 \}$. We have that the complexes $X_2$ and $2X_2$ are charged at $\mathbf{X}$ but the complex $X_1 + X_2$ is not charged at $\mathbf{X}$. The state $\mathbf{X}$ reacts to $\mathbf{X}' = \{ \mathbf{X}_1 = 1, \mathbf{X}_2 = 2 \}$ through either the reaction $2X_2 \to X_1 + X_2$ or $X_2 \to X_1$ (i.e. $\mathbf{X} \to \mathbf{X}'$). We can furthermore see that the states $\mathbf{X}'' = \{ \mathbf{X}_1 = 2, \mathbf{X}_2 = 1 \}$ and $\mathbf{X}''' = \{ \mathbf{X}_1 = 3, \mathbf{X}_2 = 0 \}$ are reachable from $\mathbf{X}$ since $\mathbf{X} \to \mathbf{X}' \to \mathbf{X}'' \to \mathbf{X}'''$ (i.e. $\mathbf{X} \leadsto \mathbf{X}'''$). Since no reactions are charged at $\mathbf{X}'''$, however, we may not return to any of $\mathbf{X}$, $\mathbf{X}'$, or $\mathbf{X}''$. These states are therefore transient, as are all of the source complexes and reactions. It follows that the CRN has an extinction event on $\mathcal{Y} = \{ X_1 + X_2, 2X_2, X_2 \}$. In fact, repeated application of $X_2 \to X_1$ leads to an extinction event regardless of initial state, so that the CRN has a guaranteed extinction event on $\mathcal{Y}$.
\end{example}

\subsection{Domination-Expanded Reaction Networks}

We briefly restate the key extensions to CRNs made in \cite{J-A-C-B}. We start with the following, which was adapted from \cite{A-E-J} and \cite{Brijder}.

\begin{definition}
\label{complexdomination}
Consider two complexes of a CRN, $y_i, y_j \in \mathcal{C}$, $y_i \not= y_j$. We say that $y_i$ \textbf{dominates} $y_j$ if $y_j \leq y_i$. We define the \textbf{domination set} of a CRN to be
\begin{equation}
\label{dominationset}
\mathcal{D}^* = \left\{ (y_i,y_j) \in \mathcal{C} \times \mathcal{C} \; | \; y_j \leq y_i, \; y_i \not= y_j \right\}.
\end{equation}
\end{definition}

\noindent The motivation underlying the domination relations $(y_i,y_j) \in \mathcal{D}$ is that, if there is sufficient molecularity in the discrete state space CRN for a reaction from $y_i$ to occur from state $\mathbf{X} \in \mathbb{Z}_{\geq 0}^m$, then there is necessarily sufficient molecularity for a reaction from $y_j$ to occur from $\mathbf{X}$. We use the domination relations to construct the following.

\begin{definition}
\label{dominationnetwork}
We say that $(\mathcal{S},\mathcal{C},\mathcal{R} \cup \mathcal{D})$ is a \textbf{domination-expanded reaction network} (dom-CRN) of the CRN $(\mathcal{S},\mathcal{C},\mathcal{R})$ if $\mathcal{D} \subseteq \mathcal{D}^*$. Furthermore, we say a dom-CRN is \textbf{$\mathcal{Y}$-admissible} if, given an absorbing complex set $\mathcal{Y} \subseteq \mathcal{C}$ on the dom-CRN, we have: (i) $\mathcal{R} \cap \mathcal{D} = \O$, and (ii) $(y_i, y_j) \in \mathcal{D}$ implies $y_j \not\in \mathcal{Y}$.
\end{definition}

\noindent In the context of dom-CRNs, we will also represent the domination relations $(y_i,y_j)$ as $y_i \to y_j$. Note that reaction goes from the dominating complex to the dominated complex, i.e. $y_j \leq y_i$ implies $y_i \to y_j$ is added to the CRN.

On a dom-CRN, we are interested in the following subgraphs.

\begin{definition}
\label{forest1}
Consider a CRN $(\mathcal{S},\mathcal{C},\mathcal{R})$ and a $\mathcal{Y}$-admissible dom-CRN $(\mathcal{S},\mathcal{C},\mathcal{R} \cup \mathcal{D})$ where $\mathcal{Y} \subseteq \mathcal{C}$ is an absorbing complex set on the dom-CRN. Then the network $(\mathcal{S},\mathcal{C},\mathcal{R}_F \cup \mathcal{D}_F)$ where $\mathcal{R}_F \subseteq \mathcal{R}$ and $\mathcal{D}_F \subseteq \mathcal{D}$ is called an \textbf{$\mathcal{Y}$-exterior forest} if, for every complex $y_i \not\in \mathcal{Y}$, there is a unique path in $\mathcal{R}_F \cup \mathcal{D}_F$ from $y_i$ to $\mathcal{Y}$.
\end{definition}

\noindent $\mathcal{Y}$-external forests define a flow from the distal portion of the reaction graph to the absorbing complex set $\mathcal{Y}$ in the dom-CRN. By convention, $\mathcal{Y}$-exterior forests will include all $\mathcal{Y}$-interior reactions of the dom-CRN.

We will be interested in the following property on $\mathcal{Y}$-exterior forests.

\begin{definition}
\label{balancing}
Consider a CRN $(\mathcal{S},\mathcal{C},\mathcal{R})$ and a $\mathcal{Y}$-admissible dom-CRN $(\mathcal{S},\mathcal{C},\mathcal{R} \cup \mathcal{D})$ where $\mathcal{Y} \subseteq \mathcal{C}$ is an absorbing complex set on the dom-CRN. Let $d = |\mathcal{D}|$. Then a $\mathcal{Y}$-exterior forest $(\mathcal{S},\mathcal{C},\mathcal{R}_F \cup \mathcal{D}_F)$ is said to be \textbf{balanced} if there is a vector $\alpha = (\alpha_R,\alpha_D) \in \mathbb{Z}_{\geq 0}^{r+d}$ with $\alpha_k>0$ for at least one $\mathcal{Y}$-exterior reaction which satisfies:
\begin{enumerate}
\item
supp$(\alpha_R) \subseteq$ supp$(\mathcal{R}_F)$ and supp$(\alpha_D) \subseteq$ supp$(\mathcal{D}_F)$;
\item
$\alpha_R \in \ker(\Gamma)$; and
\item
for every $R_k \in \mathcal{R}_F \cup \mathcal{D}_F$ where $y_{\rho(k)} \not\in \mathcal{Y}$, we have $\displaystyle{\alpha_k \geq \mathop{\sum_{i =1}^{r+d}}_{y_{\rho'(i)} = y_{\rho(k)}} \alpha_i}$.
\end{enumerate}
Otherwise, the $\mathcal{Y}$-exterior forest is said to be \textbf{unbalanced}.
\end{definition}

\noindent Condition $3$ of Definition \ref{balancing} may be interpreted as stating that, for $\mathcal{Y}$-exterior complexes, the weight of the incoming edges may not be greater than the outgoing edge. Notice that Conditions $1-3$ generated a system of equalities and inequalities on the edges of the dom-CRN. The computational implementation of the conditions of Definition \ref{balancing} is the primary focus of this paper.

\begin{remark}
For simplicity, if $\mathcal{Y}$ consists only of terminal complexes of the dom-CRN, we will refer simply to \emph{admissible dom-CRNs} and \emph{external forests} with the understanding that $\mathcal{Y}$ consists only of terminal complexes.
\end{remark}

\begin{example}
\label{example3}
Reconsider the CRN in Examples \ref{example1} and \ref{example2}. We label $y_1 = X_1 + X_2$, $y_2 = 2X_2$, $y_3 = X_2$, and $y_4 = X_1$. We have the domination relations $y_3 \leq y_1$, $y_3 \leq y_2$, and $y_4 \leq y_1$ so that $\mathcal{D}^* = \{ (y_1,y_3), (y_2,y_3), (y_1,y_4) \}$. The maximal dom-CRN is
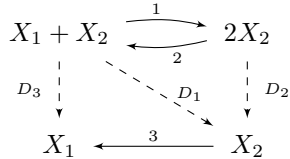
\begin{center}
\begin{tikzpicture}[auto, outer sep=3pt, node distance=2cm,>=latex']
\node (C1) {$X_1+X_2$};
\node [right of = C1, node distance = 2.5cm] (C2) {$2X_2$};
\node [below of = C2, node distance = 1.5cm] (C3) {$X_2$};
\node [left of = C3, node distance = 2.5cm] (C4) {$X_1$};
\path[->, bend left = 10] (C1) edge node [above left = -0.15cm] {\tiny $1$} (C2);
\path[->, bend left = 10] (C2) edge node [below right = -0.15cm] {\tiny $2$} (C1);
\path[->] (C3) edge node [above = -0.15cm] {\tiny $3$} (C4);
\path[dashed,->] (C1) edge [right] node {\tiny $D_1$} (C3);
\path[dashed,->] (C1) edge [left] node {\tiny $D_3$} (C4);
\path[dashed,->] (C2) edge [right] node {\tiny $D_2$} (C3);
\end{tikzpicture}
\end{center}
This dom-CRN is not admissible since the domination reaction $X_1 + X_2 \stackrel{D_3}{\longrightarrow} X_1$ leads to the terminal complex $X_1$ in the dom-CRN. Consider instead the folowing submaximal dom-CRN, where $\mathcal{D} = \{ (y_1,y_3), (y_2,y_3) \}$:
\begin{center}
\begin{tikzpicture}[auto, outer sep=3pt, node distance=2cm,>=latex']
\node (C1) {$X_1+X_2$};
\node [right of = C1, node distance = 2.5cm] (C2) {$2X_2$};
\node [below of = C2, node distance = 1.5cm] (C3) {$X_2$};
\node [left of = C3, node distance = 2.5cm] (C4) {$X_1$};
\path[->, bend left = 10] (C1) edge node [above left = -0.15cm] {\tiny $1$} (C2);
\path[->, bend left = 10] (C2) edge node [below right = -0.15cm] {\tiny $2$} (C1);
\path[->] (C3) edge node [above = -0.15cm] {\tiny $3$} (C4);
\path[dashed,->] (C1) edge [right] node {\tiny $D_1$} (C3);
\path[dashed,->] (C2) edge [right] node {\tiny $D_2$} (C3);
\end{tikzpicture}
\end{center}
Since this dom-CRN does not contain any domination reactions which lead directly to the terminal complex $X_1$, it is admissible. There are several options for external forests on this dom-CRN, including those indicated as follows in bold red:
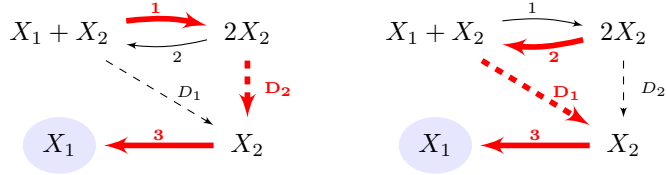
\begin{center}
\begin{tikzpicture}[auto, outer sep=3pt, node distance=2cm,>=latex']
\node (C11) {$X_1+X_2$};
\node [right of = C11, node distance = 2.5cm] (C21) {$2X_2$};
\node [below of = C21, node distance = 1.5cm] (C31) {$X_2$};
\node [left of = C31, node distance = 2.5cm,ellipse,fill=blue!10] (C41) {$X_1$};
\node [right of = C21, node distance = 2.5cm] (C12) {$X_1+X_2$};
\node [right of = C12, node distance = 2.5cm] (C22) {$2X_2$};
\node [below of = C22, node distance = 1.5cm] (C32) {$X_2$};
\node [left of = C32, node distance = 2.5cm,ellipse,fill=blue!10] (C42) {$X_1$};
\path[red,line width=0.75mm,->, bend left = 10] (C11) edge node [above left = -0.15cm] {\tiny $\mathbf{1}$} (C21);
\path[->, bend left = 10] (C21) edge node [below right = -0.15cm] {\tiny $2$} (C11);
\path[red,line width=0.75mm,->] (C31) edge node [above = -0.15cm] {\tiny $\mathbf{3}$} (C41);
\path[dashed,->] (C11) edge [right] node {\tiny $D_1$} (C31);
\path[red,line width=0.75mm,dashed,->] (C21) edge [right] node {\tiny $\mathbf{D_2}$} (C31);
\path[->, bend left = 10] (C12) edge node [above left = -0.15cm] {\tiny $1$} (C22);
\path[red,line width=0.75mm,->, bend left = 10] (C22) edge node [below right = -0.15cm] {\tiny $\mathbf{2}$} (C12);
\path[red,line width=0.75mm,->] (C32) edge node [above = -0.15cm] {\tiny $\mathbf{3}$} (C42);
\path[red,line width=0.75mm,dashed,->] (C12) edge [right] node {\tiny $\mathbf{D_1}$} (C32);
\path[dashed,->] (C22) edge [right] node {\tiny $D_2$} (C32);
\end{tikzpicture}
\end{center}
Note that, within the external forests above (bold red), there is a unique path from every complex to the terminal complex $X_1$ (shaded blue). In order to be balanced, we need to find a vector $\alpha = ((\alpha_R)_1,(\alpha_R)_2, (\alpha_R)_3, (\alpha_D)_1, (\alpha_D)_2) \in \mathbb{Z}_{\geq 0}^{5}$, $\alpha \not= \mathbf{0}$, satisying the conditions of Definition \ref{balancing}. For the external forest on the left, we have the following system of equalities and inequalities:
\begin{equation}
\label{inequalities1}
\left\{ \; \; \begin{split} (\mbox{Cond. } 1): \; \; & \hspace{0.17in} (\alpha_R)_2 = 0, \; (\alpha_D)_1 = 0 \\ (\mbox{Cond. } 2): \; \; & -(\alpha_R)_1 + (\alpha_R)_2 + (\alpha_R)_3 = 0 \\ & \hspace{0.17in} (\alpha_R)_1 - (\alpha_R)_2 - (\alpha_R)_3 = 0 \\ (\mbox{Cond. } 3): \; \; & \hspace{0.17in} (\alpha_R)_3 \geq (\alpha_D)_2 \geq (\alpha_R)_1 \geq 0. \end{split} \right.
\end{equation}
which can be satisfied by the vector $(1,0,1,0,1)$. For the external forest on the right, we have the system
\begin{equation}
\label{inequalities2}
\left\{ \; \; \begin{split} (\mbox{Cond. } 1): \; \; & \hspace{0.17in} (\alpha_R)_1 = 0, \; (\alpha_D)_2 = 0 \\ (\mbox{Cond. } 2): \; \; & -(\alpha_R)_1 + (\alpha_R)_2 + (\alpha_R)_3 = 0 \\ & \hspace{0.17in} (\alpha_R)_1 - (\alpha_R)_2 - (\alpha_R)_3 = 0 \\ (\mbox{Cond. } 3): \; \; & \hspace{0.17in} (\alpha_R)_3 \geq (\alpha_D)_1 \geq (\alpha_R)_2 \geq 0. \end{split} \right.
\end{equation}
which has no nontrivial solution. It follows that the external forest on the left is balanced while the one on the right is unbalanced. The system (\ref{inequalities2}) corresponds to the one in the introduction.
\end{example}

\subsection{Conditions for Extinction Events}

In \cite{J-A-C-B}, the authors present and prove the following main results regarding extinction events in CRNs with discrete state spaces.

\begin{theorem}
\label{mainresult}
Consider a subconservative CRN and a $\mathcal{Y}$-admissible dom-CRN where $\mathcal{Y} \subseteq \mathcal{C}$ is an absorbing complex set on the dom-CRN.  Suppose that there is a complex $y_i \not\in \mathcal{Y}$ of the dom-CRN which is recurrent from a state $\mathbf{X} \in \mathbb{Z}_{ \geq 0}^m$ in the discrete state space CRN. Then every $\mathcal{Y}$-exterior forest of the dom-CRN is balanced.
\end{theorem}

\begin{corollary}
\label{maincorollary}
Consider a subconservative CRN and a $\mathcal{Y}$-admissible dom-CRN where $\mathcal{Y} \subseteq \mathcal{C}$ is an absorbing complex set on the dom-CRN. Suppose there is a $\mathcal{Y}$-exterior forest of the dom-CRN which is unbalanced. Then the discrete state space CRN has a guaranteed extinction event on $\mathcal{Y}^c$.
\end{corollary}

\noindent When algorithmically establishing that a CRN on a discrete state space has an extinction event, we will use Corollary \ref{maincorollary} (the contrapositive of Theorem \ref{mainresult}). Note that it is sufficient to find a single $\mathcal{Y}$-exterior forest which is unbalanced, even if the dom-CRN has other balanced $\mathcal{Y}$-exterior forests.

\begin{example}
Reconsider the CRN from the introduction, which was repeated in Examples \ref{example1}, \ref{example2}, and \ref{example3}. The CRN is conservative (and therefore subconservative) and in Example \ref{example3} we established that it has an unbalanced exterior forest on an admissible dom-CRN. It follows by Corollary \ref{maincorollary} that the CRN has a guaranteed extinction event. This confirms what we observed in Example \ref{example2}.
\end{example}

\section{Computational Implementation}
\label{computationalsection}

In this section, we outline a computational algorithm capable of affirming whether a given CRN exhibits an extinction event according to Corollary \ref{maincorollary}. The program is written in Python and utilizes mixed-integer linear program (MILP) modules. MILP algorithms have been used increasingly within CRNT in recent years to verify a variety of structural properties of CRNs \cite{J-P-D,Gabor2015,Rudan2014,J-S6,Sz2,J-S4,J4}.

For brevity, we have placed technical discussion of the algorithm in the Appendix. In Appendix \ref{appendixa}, we present the background theoretical results required for the program to check the technical conditions of Corollary \ref{maincorollary}. In Appendix \ref{appendixb}, we present details on how the algorithm cycles through and expands the absorbing complex set $\mathcal{Y}$. In Appendix \ref{appendixc}, we present a detailed description of the modules outlined in Section \ref{algorithmsection}. In Appendix \ref{appendixd}, we summarize the output of our BioModels Database run.


\subsection{Description of Program}
\label{algorithmsection}

\noindent The pseudocode given in Algorithm \ref{alg:implementation} implements Corollary \ref{maincorollary}. The modules are briefly described below. More detailed explanations are given in Appendix \ref{appendixc}.

\begin{algorithm}
\begin{algorithmic}[1]
\Procedure{DiscreteExtinction}{$N$}
\State CreateModel($N$)
\If{IsSubconservative($N$)}
\State DiscreteExtinction $\gets$ \texttt{false}
\State $\mathcal{D}^*$ = DominationSet($N$)
\State $\mathcal{Y}^*$ = $\O$
\State $(\mathcal{Y},\mathcal{D})$ = FindAdmissibleDom($N$,$\mathcal{Y}^*$,$\mathcal{D}^*$)
\While{DiscreteExtinction $=$ \texttt{false} \textbf{and} $\mathcal{Y} \not= \mathcal{C}$}
\State dom-$N$ = DominationExpandedNetwork($N$,$\mathcal{D}$)
\ForAll{$F$ in CycleForests(dom-$N$)}
\If{IsExtForest($F$)}
\If{IsBalanced($F$)}
\State $\mathcal{Y}^*$ $\gets$ ExpandedY($F$)
\State $(\mathcal{Y},\mathcal{D})$ = FindAdmissibleDom($N$,$\mathcal{Y}^*$,$\mathcal{D}^*$)
\Else
\State DiscreteExtinction $\gets$ \texttt{true}
\EndIf
\EndIf
\EndFor
\EndWhile
\EndIf
\State WriteOutput
\EndProcedure
\end{algorithmic}
\caption{Pseudocode implementation of Corollary \ref{maincorollary} for input CRN $N = (\mathcal{S},\mathcal{C},\mathcal{R})$.}
\label{alg:implementation}
\end{algorithm}

\begin{itemize}
\item[] \hspace{-0.4in} CreateModel($N$):\\This module takes in a CRN $N = (\mathcal{S},\mathcal{C},\mathcal{R})$ and generates the structural matrices $Y$, $I_a$, $\Gamma$, $I_s$, and $A$.
\item[] \hspace{-0.4in} IsSubconservative($N$):\\This module determines whether the CRN is subconservative. If the CRN is subconservative, it further runs the module IsConservative($N$) to determine whether it is also conservative.
\item[] \hspace{-0.4in} DominationSet($N$):\\This module generates the domination set $\mathcal{D}^*$ (Definition \ref{complexdomination}).
\item[] \hspace{-0.4in} FindTerm($N$):\\This module determines the terminal complexes of a given CRN (or dom-CRN).
\item[] \hspace{-0.4in} FindAdmissibleDom($N$,$\mathcal{Y}^*$,$\mathcal{D}^*$):\\This module determines the minimal absorbing complex set $\mathcal{Y} \supseteq \mathcal{Y}^*$ and maximal set $\mathcal{D} \subseteq \mathcal{D}^*$ such that the resulting dom-CRN is $\mathcal{Y}$-admissible (Definition \ref{dominationnetwork}). (See Appendix \ref{appendixb} for details.)
\item[] \hspace{-0.4in} DominationExpandedNetwork($N$,$\mathcal{D}$):\\This module generates the structural matrices $Y$, $I_a$, $\Gamma$, $I_s$, and $A$ for the admissible dom-CRN.
\item[] \hspace{-0.4in} CycleForests(dom-$N$):\\This module cycles over every combination $F$ of reactions in the dom-CRN such that each $\mathcal{Y}$-exterior complex is the source complex for exactly one reaction. By default, every $\mathcal{Y}$-interior reaction of the dom-CRN is included in every set $F$.
\item[] \hspace{-0.4in} IsExtForest($F$):\\This module checks whether the set of reactions $F$ is a $\mathcal{Y}$-exterior forest (Definition \ref{forest1}) according to Theorem \ref{foresttheorem} (Appendix \ref{appendixa}).
\item[] \hspace{-0.4in} IsBalanced($F$):\\This module checks whether $F$ is balanced (Definition \ref{balancing}) according to Theorem \ref{tieringtheorem} (Appendix \ref{appendixa}).
\item[] \hspace{-0.4in} ExpandedY($F$):\\If a given exterior forest $F$ is balanced but the balancing vector $\alpha \in \mathbb{R}^{r + d}_{\geq 0}$ does not have support on $\mathcal{R}_F \cup \mathcal{D}_F$, this module expands the absorbing complex set $\mathcal{Y}$. (See Appendix \ref{appendixb} for details.)
\item[] \hspace{-0.4in} WriteOutput:\\This module writes the output into a {\tt .dat} file.

\end{itemize}



\subsection{BioModels Database}

We ran Algorithm \ref{alg:implementation} on 458 curated models from the European Bioinformatics Institute's BioModels Database. In total, 86 models were found which had a guaranteed extinction event according to Corollary \ref{maincorollary}. Of these, 53 models were conservative and 33 were subconservative but not conservative. We limited our run to those models which contained $50$ or fewer reactions.

We further classify the models which exhibit an extinction event according to the following:
\begin{enumerate}
\item
A species $X_i$ is classified as \emph{source only} if, in the unbalanced $\mathcal{Y}$-exterior forest which guarantees extinction, $X_i$ appears only in source complexes.
\item
A species $X_i$ is classified as \emph{product only} if, in the unbalanced $\mathcal{Y}$-exterior forest which guarantees extinction, $X_i$ appears only in product complexes.
\end{enumerate}
These classifications identify mechanisms whose primary purpose is to convert one substrate (source) into another (product). In such mechanisms, we expect discrete extinction events to occur as a consequence of a source being used up and/or a product being formed. For example, consider the classical Michaelis-Menten mechanism \cite{M-M}:
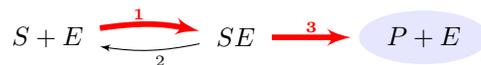
\begin{center}
\begin{tikzpicture}[auto, outer sep=3pt, node distance=2cm,>=latex']
\node (C1) {$S+E$};
\node [right of = C1, node distance = 2.5cm] (C2) {$SE$};
\node [right of = C2, node distance = 2.5cm,ellipse, fill = blue!10] (C3) {$P+E$};
\path[red,line width=0.75mm,->, bend left = 10] (C1) edge node [above left = -0.15cm] {\tiny $\mathbf{1}$} (C2);
\path[->, bend left = 10] (C2) edge node [below right = -0.15cm] {\tiny $2$} (C1);
\path[red,line width=0.75mm,->] (C2) edge node [above = -0.15cm] {\tiny $\mathbf{3}$} (C3);
\end{tikzpicture}
\end{center}
where the unbalanced exterior forest is indicated in bold red, and terminal complex is shaded blue. The CRN is conservative, and has the source only species $S$ and product only species $P$. Notice that $S$ is source only even though it appears as a product in the original CRN since the reaction $SE \to S +E$ is not contained in the exterior forest (red). The remaining species $E$ and $SE$ appear as both a source and product in the exterior forest. For large CRNs, the property of having source only or product only species may be difficult to verify directly.

Of the 86 models classified as having a discrete extinction event from the BioModels run, 81 have source only species, 57 have product only species, and 52 have both source only and product only species. These results are consistent with the observation that many models in the database model processive biochemical mechanisms like the Michaelis-Menten mechanism. A full enumeration and classification of the BioModels networks identified by the algorithm as having a guaranteed extinction event is given in Appendix \ref{appendixd}.

The results of the BioModels run has suggested some interesting CRNs with extinction events and some unique pathways for obtaining extinction. It also suggests avenues for future work. We now further investigate two models which the program identified as having an extinction event according to Corollary \ref{maincorollary}: a model of polyamine metabolism in mammals ({\tt biomd0000000190}) \cite{Rodriguez2006}, and a model of the pentose phosphate pathway in \emph{Trypanosoma brucei} ({\tt biomd0000000513}) \cite{Bakker1997,Kerkhoven2012,Kerkhoven2013}.

\subsection{Polyamine Metabolism}
\label{polyaminesection}

Consider the model of polyamine metabolism given in Table \ref{table2} which corresponds to model {\tt biomd0000000190} in the BioModels database \cite{Rodriguez2006}.

\begin{table}[h]
	\centering
	\footnotesize
	\begin{tabular}{|l|l|}
		\hline
		\rowcolor[gray]{0.9}
		(1) $Met \rightarrow SAM$ &
		(7) $aS \rightarrow D$ \\
		(2) $SAM \rightarrow dcSAM$ &
		(8) $D + Acetyl\mbox{-}CoA \rightarrow aD + CoA$ \\
		\rowcolor[gray]{0.9}
		(3) $Orn \rightarrow P$ \; \; \; \; \; \; \; \; \; \; \;  &
		(9) $aD \rightarrow P$ \\
		(4) $P + dcSAM \rightarrow D$ &
		(10) $CoA \rightleftarrows Acetyl\mbox{-}CoA$ \\
		\rowcolor[gray]{0.9}
		(5) $D + dcSAM \rightarrow S$ &
		(11) $P \rightarrow \O$ \\
		(6) $S + Acetyl\mbox{-}CoA \rightarrow aS + CoA$ &
		(12) $aD \rightarrow \O$ \\
		\hline
	\end{tabular}
	\caption{Reactions for Polyamine Metabolism ($Met$ - methionine, $SAM$ - S-adenosylmethionine, $dcSAM$ - S-adenosylmethionine decarboxylase, $Orn$ - ornithine, $P$ - putrescine, $D$ - spermidine, $S$ - spermine, $aS$ - N-acetyl-spermine, $aD$ - N-acetylspermidine.)}
	\label{table2}
\end{table}

The program identifies the CRN as subconservative and exhibiting an extinction event according to Corollary \ref{maincorollary}. It returns the following sets:
\[\begin{split}
\mathcal{Y} & = \{ \O, dcSAM, S, D, CoA, Acetyl\mbox{-}CoA\} \\
\mathcal{D} & = \{ aS + CoA \to aS, aD +CoA \to aD, P + cdSAM \to P \}.
\end{split}\]
It follows from Corollary \ref{maincorollary} that the following complexes are transient:
\[\mathcal{Y}^c = \{ Met, SAM, Orn, P + dcSAM, D + cdSAM, S + Acetyl\mbox{-}CoA, aS, D + Acetyl\mbox{-}CoA, aD\}.\]
The program further identifies $Met$ and $Orn$ as source only species which suggests that they are inputs to the system. Indeed, the results are consistent with the observation that the mechanism requires input of $Met$ and $Orn$ in order to maintain function as without these species all reactions except (12) will eventually shut down.

\begin{remark}
It is worth noting that, although the complexes in $\mathcal{Y}^c$ are guaranteed to be transient by Corollary \ref{maincorollary}, it is not necessarily the case that the complexes in $\mathcal{Y}$ are recurrent. For this example, we have $\{ S,Acetyl\mbox{-}CoA \} \subset\mathcal{Y}$ and $S + Acetyl\mbox{-}CoA \in \mathcal{Y}^c$. It cannot be the case that both $S$ and $Acetyl\mbox{-}CoA$ are recurrent since this would imply the recurrence of $S + Acetyl\mbox{-}CoA$ and therefore $S + Acetyl\mbox{-}CoA \not\in \mathcal{Y}^c$. It can inferred from the CRN that $S$ is transient despite not being an element of $\mathcal{Y}^c$. That is, while the program returned a set of transient complexes, it did not return the maximal such set.
\end{remark}

\subsection{Pentose Phosphate Pathway}
\label{pentosesection}

Consider the model for the pentose phosphate pathway in \emph{Trypanosoma brucei} given in Table \ref{table1} which corresponds to model {\tt biomd0000000513} in the BioModels Database\cite{Bakker1997,Kerkhoven2012,Kerkhoven2013}. The algorithm presented in Section \ref{algorithmsection} identifies this mechanism as having a guaranteed extinction event.

\begin{table}[h]
	\centering
	\footnotesize
	\begin{tabular}{|l|l|}
		\hline
		\rowcolor[gray]{0.9}
		(1) $Glc_e \rightleftarrows Glc_c$ &
		(11) $Glc_g + ATP_g \rightleftarrows Glc\mbox{-}6\mbox{-}P_g + ADP_g$ \\
		(2) $Glc_c \rightleftarrows Glc_g$ &
		(12) $Fru\mbox{-}6\mbox{-}P_g + ATP_g \rightleftarrows Fru\mbox{-}1\mbox{,}6\mbox{-}BP_g+ADP_g$ \\
		\rowcolor[gray]{0.9}
		(3) $Glc\mbox{-}6\mbox{-}P_g \rightleftarrows Fru\mbox{-}6\mbox{-}P_g$ \; \; \; \; \; \; \; \; \; \; \;  &
		(13) $Fru\mbox{-}1\mbox{,}6\mbox{-}BP_g \rightleftarrows DHAP_g + GA\mbox{-}3\mbox{-}P_g$ \\
		(4) $Gly\mbox{-}3\mbox{-}P_c \rightarrow DHAP_c$ &
		(14) $DHAP_c + Gly\mbox{-}3\mbox{-}P_g \rightleftarrows DHAP_g + Gly\mbox{-}3\mbox{-}P_c$ \\
		\rowcolor[gray]{0.9}
		(5) $DHAP_g \rightleftarrows GA\mbox{-}3\mbox{-}P_g$ &
		(15) $DHAP_g + NADH_g \rightleftarrows NAD^+_g + Gly\mbox{-}3\mbox{-}P_g$ \\
		(6) $3\mbox{-}PGA_c \rightleftarrows 2\mbox{-}PGA_c$ &
		(16) $Gly\mbox{-}3\mbox{-}P_g + ADP_g \rightleftarrows Gly_e + ATP_g$ \\
		\rowcolor[gray]{0.9}
		(7) $3\mbox{-}PGA_g \rightleftarrows 3\mbox{-}PGA_c$ &
		(17) $GA\mbox{-}3\mbox{-}P_g + NAD^+_g + Pi_g \rightleftarrows 1\mbox{,}3\mbox{-}BPGA_g + NADH_g$ \\
		(8) $2\mbox{-}PGA_c \rightleftarrows PEP_c$ &
		(18) $1\mbox{,}3\mbox{-}BPGA_g + ADP_g \rightleftarrows 3\mbox{-}PGA_g + ATP_g$ \\
		\rowcolor[gray]{0.9}
		(9) $Pyr_c \rightarrow Pyr_e$ &
		(19) $PEP_c + ATP_c \rightleftarrows Pyr_c + ADP_c$ \\
		(10) $ATP_c \rightarrow ADP_c$ &
		(20) $2 ADP_c \rightleftarrows ATP_c + AMP_c$ \\
		\rowcolor[gray]{0.9}
		&
		(21) $2ADP_g \rightleftarrows ATP_g + AMP_g$ \\
		\hline
	\end{tabular}
	\caption{Reactions for pentose phosphate pathway}
	\label{table1}
\end{table}

The algorithm identifies several structurally distinct pathways by which an extinction event can occur. The first is the following subnetwork of the dom-CRN, which incorporates reactions (9), (10), (19), and (20) as indexed in Table \ref{table1}:
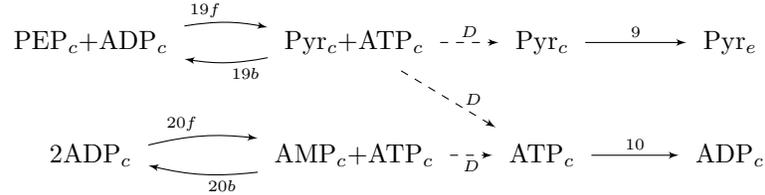
\begin{center}
\begin{tikzpicture}[auto, outer sep=3pt, node distance=2cm,>=latex']
\node (C1) {PEP$_c$+ADP$_c$};
\node [right of = C1, node distance = 3.5cm] (C2) {Pyr$_c$+ATP$_c$};
\node [right of = C2, node distance = 2.5cm] (C3) {Pyr$_c$};
\node [right of = C3, node distance = 2.5cm] (C4) {Pyr$_e$};
\node [below of = C1, node distance = 1.5cm] (C5) {$2$ADP$_c$};
\node [below of = C2, node distance = 1.5cm] (C6) {AMP$_c+$ATP$_c$};
\node [below of = C3, node distance = 1.5cm] (C7) {ATP$_c$};
\node [below of = C4, node distance = 1.5cm] (C8) {ADP$_c$};
\path[->, bend left = 10] (C1) edge node [above left = -0.15cm] {\tiny $19f$} (C2);
\path[->, bend left = 10] (C2) edge node [below right = -0.15cm] {\tiny $19b$} (C1);
\path[->] (C3) edge node [above = -0.15cm] {\tiny $9$} (C4);
\path[->, bend left = 10] (C5) edge node [above left = -0.15cm] {\tiny $20f$} (C6);
\path[->, bend left = 10] (C6) edge node [below right = -0.15cm] {\tiny $20b$} (C5);
\path[->] (C7) edge node [above = -0.15cm] {\tiny $10$} (C8);
\path[dashed,->] (C6) edge node [below = -0.15cm] {\tiny $D$} (C7);
\path[dashed,->] (C2) edge node [right = -0cm] {\tiny $D$} (C7);
\path[dashed,->] (C2) edge node [above = -0.15cm] {\tiny $D$} (C3);
\end{tikzpicture}
\end{center}
The program identifies the complexes PEP$_c$+ADP$_c$, Pyr$_c$+ATP$_c$, Pyr$_c$, $2$ADP$_c$, AMP$_c+$ATP$_c$, and ATP$_c$ as transient. It follows immediately that, at any extinction state, we have the following counts: Pyr$_c = 0$, ADP$_c \leq 1$, and ATP$_c = 0$. It can be furthermore seen from reactions (10) and (20) that ADP$_c = 0$ is not absorbing, so that ADP$_c \geq 1$ infinitely often. It follows that ADP $ =1$ at the extinction states, from which it follows from the transience of PEP$_c$ + ADP$_c$ that PEP$_c = 0$ at the extinction states.

We now construct the pathway to extinction. It follows from the conservation on the original CRN and the observation that Pyr$_e$ only appears as a product in any reaction that any reaction which forms Pyr$_e$ must have a final occurrence. Furthermore, we can exhaust PEP$_c$ by converting it into Pyr$_c$ through the forward reaction in (19) since ADP$_c \geq 1$ infinitely often, as previously argued. A less trivial pathway to extinction occurs for the species ADP$_c$ and ATP$_c$. The conservation
\[\mbox{AMP}_c + \mbox{ADP}_c + \mbox{ATP}_c = \mbox{total}_A\]
suggests that the system may become locked by converting ADP$_c$ into AMP$_c$ and ATP$_c$ through reaction (20), and then converting ATP$_c$ into ADP$_c$ through reaction (10), and then repeating as many times as possible. Eventually we will arrive at a state where ADP$_c = 0$, AMP$_c = \mbox{total}_A-1$, and ATP$_c =1$, after which reaction (10) locks both reactions. We must, however, consider the possible that ATP$_c$ is converted to ADP$_c$ by another pathway and, in fact, the backward reaction in (19) is exactly such a reaction. It is also, however, the only such pathway. Consequently reactions (9) and (19) must shut down before reactions (10) and (20) can be shut down. The complete sequence of reactions required to shut down the indicated complexes is therefore:
\begin{enumerate}
\item
Convert all possible substrates in PEP$_c$ or Pyr$_c$ and define total$_P = $PEP$_c+$Pyr$_c$.
\item
Fire the forward reaction in (19) to convert all PEP$_c$ into Pyr$_c$, replenishing ADP$_c$ through reaction (10) as required.
\item
Fire reaction (9) to convert all Pyr$_c$ into Pyr$_e$.
\item
Fire reaction (10) to convert all ATP$_c$ into ADP$_c$.
\item
Fire the forward reaction in (20) to convert all ADP$_c$ into AMP$_c$ and ATP$_c$.
\item
Repeat steps 4. and 5. until you arrive at the state ADP$_c = 0$, AMP$_c = \mbox{total}_A - 1$, and ATP$_c = 1$, and then fire (10).
\end{enumerate}
The final extinction state is PEP$_c = 0$, Pyr$_c = 0$, Pyr$_e = \mbox{total}_P$, AMP$_c = \mbox{total}_A$, ADP$_c = 1$, and ATP$_c = 0$. Notice that the extinction of PEP$_c$ is only guaranteed for all trajectories by the observation that ADP$_c \geq 1$ infinitely often. We also require that no more PEP$_c$ and Pyr$_c$ can be produced by the remaining pathways. This example illustrates that the program is able to identify transient complexes which might be very difficult to determine by direct analysis of potential pathways to extinction.

The other subnetwork which was identified as leading to extinction is the following, which incorporates reactions (4) and (14) in Table \ref{table1}:
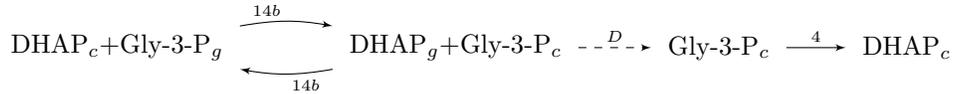
\begin{center}
\begin{tikzpicture}[auto, outer sep=3pt, node distance=2cm,>=latex']
\node (C1) {DHAP$_c$+Gly\mbox{-}3\mbox{-}P$_g$};
\node [right of = C1, node distance = 4.5cm] (C2) {DHAP$_g$+Gly\mbox{-}3\mbox{-}P$_c$};
\node [right of = C2, node distance = 3.5cm] (C3) {Gly\mbox{-}3\mbox{-}P$_c$};
\node [right of = C3, node distance = 2.5cm] (C4) {DHAP$_c$};
\path[->, bend left = 10] (C1) edge node [above left = -0.15cm] {\tiny $14b$} (C2);
\path[->, bend left = 10] (C2) edge node [below right = -0.15cm] {\tiny $14b$} (C1);
\path[->] (C3) edge node [above = -0.15cm] {\tiny $4$} (C4);
\path[dashed,->] (C2) edge node [above = -0.15cm] {\tiny $D$} (C3);
\end{tikzpicture}
\end{center}
The program identifies the complexes DHAP$_c+$Gly$\mbox{-}3\mbox{-}$P$_g$, DHAP$_g+\mbox{Gly-}3\mbox{-}$P$_c$, and $\mbox{Gly-}3\mbox{-}$P$_c$ as transient. It follows that at the extinction state we have $\mbox{Gly-}3\mbox{-}$P$_c = 0$. It furthermore follows from the observation that DHAP$_c > 0$ at the extinction state that we must have Gly$\mbox{-}3\mbox{-}$P$_g = 0$. This is consistent with the observation that, after firing the forward reaction in (14) and then reaction (4) repeatedly to exhaust $\mbox{Gly-}3\mbox{-}$P$_g$ and $\mbox{Gly-}3\mbox{-}$P$_c$,  there is no mechanism by which to convert DHAP$_c$ or any of its derivatives back into $\mbox{Gly-}3\mbox{-}$P or any of its derivatives.

\begin{remark}
This example once again demonstrates that the set of transient complexes returned is not necessarily the maximal such set. The species PEP$_c$ is necessarily zero at any extinction state; however, PEP$_c$ appears as its own complex in Table \ref{table1} and this complex was not identified as transient by the program. From reactions (6), (7), and (8), we can furthermore identify the complexes $\mbox{PEP}_c$, 2-PGA$_c$, 3-PGA$_c$, and 3-PGA$_g$ as transient since they can all be transformed in Pyr$_e$. We have that the set of complexes which is necessarily transient at the extinction states is greater than that strictly guaranteed by Condition \ref{maincorollary}. 
\end{remark}

\section{Conclusions}
\label{conclusions}

In this paper, we have presented a computational implementation of the conditions derived in the companion paper \cite{J-A-C-B} for an extinction event in a CRN with a discrete state space. We have run the program on one of the most widely studied network databases in system biology, the European Bioinformatics Institute's BioModels database. This work has yielded some mathematically and biologically interesting pathways by which extinction in biological systems may be attained, as indicated by our analysis of the model of polyamine metablism in mammals and the model of the pentose phosphate pathway in \emph{Trypanosoma brucei}.

The most notable avenue for future work opened up by the study is in extending the current results to determine the maximal set of transient complexes. Currently, even when Algorithm \ref{alg:implementation} guarantees an extinction event by Corollary \ref{maincorollary}, it does not necessary give the maximal such set. For example, for the polyamine metabolism model studied in Section \ref{polyaminesection}, the species $S$ is transient but was not identified as such. For the pentose phosphate pathway model studied in Section \ref{pentosesection}, PEP$_c$ is transient but not identified as such. Further work will investigate methods for expanding the transient complex set given by Corollary \ref{maincorollary} and Algorithm \ref{alg:implementation} into the maximal such set.


\section*{Acknowledgments}

The author was supported by  Army Research Office grant W911NF-14-1-0401 and the Henry Woodward Fund. The author is grateful to D. Anderson, G. Craciun, and R. Brijder for insightful comments and corrections during the early stages of this work, and E. Tonello for granting access to her superb CRN Python libraries.

\newpage




\begin{appendices}

\section{Background Results}
\label{appendixa}

We require the following results in order to computationally implement the modules of Algorithm \ref{alg:implementation}. The proofs of Theorem \ref{foresttheorem} and Theorem \ref{tieringtheorem} can be found later in this Appendix. The proof of Theorem \ref{kerneltheorem} can be found in the Appendix of \cite{H-F}.

\begin{theorem}
\label{foresttheorem}
Consider a CRN $(\mathcal{S},\mathcal{C},\mathcal{R})$ and a $\mathcal{Y}$-admissible dom-CRN $(\mathcal{S},\mathcal{C},\mathcal{R} \cup \mathcal{D})$ where $\mathcal{Y} \subseteq \mathcal{C}$ is an absorbing complex set on the dom-CRN. Let $I_{a}$ denote the adjacency matrix of the dom-CRN. Consider a subnetwork of the dom-CRN $(\mathcal{S},\mathcal{C},\mathcal{R}_F \cup \mathcal{D}_F)$ where $\mathcal{R}_F \subseteq \mathcal{R}$ and $\mathcal{D}_F \subseteq \mathcal{D}$, and suppose that, for all $y_i \not\in \mathcal{Y}$, $y_i$ is the source for exactly one reaction in the subnetwork $(\mathcal{S},\mathcal{C},\mathcal{R}_F \cup \mathcal{D}_F)$. Then the following are equivalent:
\begin{enumerate}
\item
The subnetwork $(\mathcal{S},\mathcal{C},\mathcal{R}_F \cup \mathcal{D}_F)$ is an $\mathcal{Y}$-exterior forest.
\item
There are no cycles in the $\mathcal{Y}$-exterior portion of the reaction graph of $(\mathcal{S},\mathcal{C},\mathcal{R}_F \cup \mathcal{D}_F)$.
\item
There is a vector $\mathbf{v} = (\mathbf{v}_R,\mathbf{v}_D) \in \mathbb{R}_{\geq 0}^{r+d}$ such that $\mbox{supp}(\mathbf{v}_R) = \mbox{supp}(\mathcal{R}_F)$, $\mbox{supp}(\mathbf{v}_D) = \mbox{supp}(\mathcal{D}_F)$, and $[I_a \mathbf{v}]_j < 0$ for every $j \in \{1, \ldots, n\}$ such that $y_j \not\in \mathcal{Y}$.
\end{enumerate}
\end{theorem}


\begin{theorem}
\label{tieringtheorem}
Consider a CRN $(\mathcal{S},\mathcal{C},\mathcal{R})$, $\mathcal{Y}$-admissible dom-CRN $(\mathcal{S},\mathcal{C},\mathcal{R} \cup \mathcal{D})$, and $\mathcal{Y}$-exterior forest $(\mathcal{S},\mathcal{C},\mathcal{R}_F \cup \mathcal{D}_F)$ where $\mathcal{Y} \subseteq \mathcal{C}$ is an absorbing complex set on the dom-CRN. Let $I_{a}$ denote the adjacency matrix for the dom-CRN. Then a vector $\alpha = (\alpha_R,\alpha_D) \in \mathbb{R}^{r+d}_{\geq 0}$ with $\mbox{supp}(\alpha_R) \subseteq \mbox{supp}(\mathcal{R}_F)$ and $\mbox{supp}(\alpha_D) \subseteq \mbox{supp}(\mathcal{D}_F)$ satisfies Condition $3$ of Definition \ref{balancing} if and only if $[I_a \alpha]_{j} \leq 0$ for every $j \in \{1, \ldots, n\}$ such that $y_j \not\in \mathcal{Y}$.
\end{theorem}

\begin{theorem}
\label{kerneltheorem}
Consider a CRN with terminal SLCs $\mathcal{T} = \{ T^{(1)}, \ldots, T^{(t)} \}$. Then the basis of ker$(A)$ consists of vectors $\mathbf{b}^{(i)} \in \mathbb{R}_{\geq 0}^n$ where
\[\mbox{supp}\left(\mathbf{b}^{(i)}\right) = \mbox{supp}\left(\mathcal{T}^{(i)}\right), \; \; \; i \in \{1, \ldots, t\}.\]
\end{theorem}

\begin{proof}[Proof of Theorem \ref{foresttheorem}]
We prove Condition $1 \Leftrightarrow$ Condition $2$ $\Leftrightarrow$ Condition $3$.\\

\noindent \emph{Condition $1 \Rightarrow$ Condition $2$}: Suppose that $(\mathcal{S},\mathcal{C},\mathcal{R}_F \cup \mathcal{D}_F)$ is a $\mathcal{Y}$-exterior forest and that there is a cycle in $\mathcal{Y}$-exterior portion of the reaction graph. Since every $\mathcal{Y}$-exterior complex has a path to $\mathcal{Y}$, there is a complex in the cycle which has a reaction leading out of the cycle. From this complex, however, there are two distinct paths which lead to $\mathcal{Y}$: (i) the direct path with the reaction leading out of the cycle as its first reaction; and (ii) the path which transverses the cycle first, then the path in (i). It follows that $(\mathcal{S},\mathcal{C},\mathcal{R}_F \cup \mathcal{D}_F)$ is not a $\mathcal{Y}$-exterior forest, which is a contradiction. It follows that Condition $1$ implies Condition $2$.\\

\noindent \emph{Condition $2 \Rightarrow$ Condition $1$}: Suppose there are no cycles in the $\mathcal{Y}$-exterior portion of the reaction graph. Consider a $\mathcal{Y}$-exterior complex $y_i \not\in \mathcal{Y}$. By definition, there is a unique reaction $(y_i,y_j) \in \mathcal{R}_F \cup \mathcal{D}_F$ with $y_i$ as its source complex. If $y_j \in \mathcal{Y}$, we are done; if $y_j \not\in \mathcal{Y}$, however, we continue the path with the unique reaction $(y_j,y_k) \in \mathcal{R}_F \cup \mathcal{D}_F$, and so on. Since there are a finite number of $\mathcal{Y}$-exterior complexes, this process may not terminate before either reaching a $\mathcal{Y}$-interior complex, or a complex already in the path, which creates a cycle and is therefore a contradiction. It follows that the path reaches $\mathcal{Y}$ and is unique. Since $y_i \not\in \mathcal{Y}$ was chosen arbitrarily, it follows that $(\mathcal{S},\mathcal{C},\mathcal{R}_F \cup \mathcal{D}_F)$ is a $\mathcal{Y}$-exterior forest. We have shown Condition $2$ implies Condition $1$.\\

\noindent \emph{Condition $2 \Rightarrow$ Condition $3$}: Suppose there are no cycles in the $\mathcal{Y}$-exterior portion of the reaction graph. We will inductively construct the required vector $\mathbf{v} = (\mathbf{v}_R, \mathbf{v}_D) \in \mathbb{R}_{\geq 0}^{r+d}$. 

We start by setting $v_i = 0$ for $y_{\rho(i)} \to y_{\rho'(i)} \not\in \mathcal{R}_F \cup \mathcal{D}_F$. Next, we define the set:
\[I_0 = \{ i \in \{1, \ldots, r+d \} \; | \; y_{\rho(i)} \to y_{\rho'(i)} \in \mathcal{R}_F \cup \mathcal{D}_F \mbox{ where } y_{\rho(i)} \not\in \mathcal{Y} \mbox{ and } y_{\rho'(i)} \in \mathcal{Y}\}.\]
We set $V_0 = 1$ and $v_i = V_0$ for all $i \in I_0$. That is, reactions which are not contained in the $\mathcal{Y}$-exterior forest are assigned a weight of $0$ in the vector $\mathbf{v}$, and those which lead directly into $\mathcal{Y}$ from the $\mathcal{Y}$-exterior portion are assigned a weight of $1$.

We now inductively work out from $\mathcal{Y}$ in the $\mathcal{Y}$-exterior forest. We define the following:
\[ I_k = \{ i \in \{1, \ldots, r+d\} \; | \; y_{\rho(i)} \to y_{\rho'(i)} \in \mathcal{R}_F \cup \mathcal{D}_F \mbox{ where } \rho'(i) = \rho(j) \mbox{ for some } j \in I_{k-1}\}.\]
Since the set of complexes is finite, there are a finite number of such sets which are nonempty: $I_0, I_1, \ldots, I_{n_I -1}, I_{n_I}$. Further, since every $\mathcal{Y}$-exterior complex is the source for exactly one reaction in the $\mathcal{Y}$-exterior forest and there are no cycles, each such complex is assigned to exactly one set $I_k$, $k \in \{1, \ldots, n_I\}$. It follows that this is a partition of the $\mathcal{Y}$-exterior complexes, i.e. $I_{k_1} \cap I_{k_2} = \O$ for all $k_1, k_2 \in \{1, \ldots, n_I \},$ $k_1 \not= k_2$, and $\bigcup_{k=1}^{n_I} I_k = \mbox{supp}(\mathcal{Y})$.

Now define $n_k = |I_k|$ and
\[V_k = \frac{V_{k-1}}{n_k} - \delta, \; \; k \in \{1, \ldots, n_I\}\]
where $\delta > 0$ is a constant to be a determined later. For each $i \in I_k$, we set $v_i = V_k$.

Now consider a complex $y_j \not\in \mathcal{Y}$ and the reaction $y_{\rho(i)} \to y_{\rho'(i)} \in \mathcal{R}_F \cup \mathcal{D}_F$ where $\rho(i) = j$ and $i \in I_k$. Notice that this outgoing reaction is unique while there are at most $n_{k+1}$ reactions which lead to $y_j$. We therefore have:
\[
[I_a \mathbf{v}]_{j} \leq \left( \mathop{\sum_{l=1}^{r+d}}_{l \in I_{k+1}} v_l \right) - v_{i} = \sum_{l=1}^{n_{k+1}} \left( \frac{V_k}{n_{k+1}} - \delta \right) -V_k = \left( V_k - n_{k+1}\delta \right) - V_k = -n_{k+1}\delta < 0.
\]
We now choose $\delta > 0$ so that $\frac{V_k}{n_k} - \delta > 0$ for all $k \in \{ 1, \ldots, n_I \}$. We have constructed a vector $\mathbf{v} \in \mathbb{R}^{r + d}_{\geq 0}$ with the desired properties, which shows Condition $2$ implies Condition $3$.\\

\noindent \emph{Condition $3 \Rightarrow$ Condition $2$}: Suppose also that there is a cycle on the $\mathcal{Y}$-exterior portion of the reaction graph. We denote this cycle by:
\[
y_{\mu(1)} \rightarrow y_{\mu(2)} \rightarrow \cdots \rightarrow y_{\mu(l)} \rightarrow y_{\mu(1)}.
\]
We define the following sets:
\[I_{cycle} = \{ i \in \{1, \ldots, n \} \; | \; \mu(j) = i \mbox{ for some } j \in \{ 1, \ldots, l\} \}\]
and
\[I_{in}(j) = \{ k \in \{1, \ldots, n \} \; | \; y_i \to y_j \in \mathcal{R}_F \cup \mathcal{D}_F \mbox{ where } i \not\in I_{cycle} \mbox{ and } j \in I_{cycle}\}.\]
Notice that, since every $\mathcal{Y}$-exterior complex is the source for exactly one reaction, there are no reactions which lead out of the cycle.

Define $\mathbf{v} = (\mathbf{v}_R,\mathbf{v}_D) \in \mathbb{R}_{\geq 0}^{r+d}$ such that $\mbox{supp}(\mathbf{v}_R) = \mbox{supp}(\mathcal{R}_F)$, $\mbox{supp}(\mathbf{v}_D) = \mbox{supp}(\mathcal{D}_F)$. It follows that, for $j \in \{ 1, \ldots, l\}$, we have:
\begin{equation}
\label{512}
 [I_a \mathbf{v} ]_{\mu(j)} = \left( \sum_{i \in I_{in}(\mu(j))} v_i \right) + (v_{\mu(j-1)} - v_{\mu(j)}) < 0.
\end{equation}
Taking the summation of (\ref{512}) over $j \in \{1, \ldots, l \}$ gives
\[ \sum_{j=1}^l [I_a \mathbf{v} ]_{\mu(j)} =  \left( \sum_{j=1}^l \sum_{i \in I_{in}(\mu(j))} v_i \right) +  \sum_{j=1}^l v_{\mu(j-1)} -  \sum_{j=1}^l v_{\mu(j)} = \left( \sum_{j=1}^l \sum_{i \in I_{in}(\mu(j))} v_i \right) \geq 0 \]
where the inequality follows from $\mathbf{v} \in \mathbb{R}_{\geq 0}^r$. This is inconsistent with the requirement that $[I_a \mathbf{v}]_j < 0$ for $j \in \{1, \ldots, n \}$ such that $y_j \not\in \mathcal{Y}$, so we have shown Condition $3$ implies Condition $2$.
\end{proof}

\begin{proof}[Proof of Theorem \ref{tieringtheorem}]
Consider a vector $\alpha = (\alpha_R,\alpha_D) \in \mathbb{R}^{r+d}_{\geq 0}$ with $\mbox{supp}(\alpha_R) \subseteq \mbox{supp}(\mathcal{R}_F)$ and $\mbox{supp}(\alpha_D) \subseteq \mbox{supp}(\mathcal{D}_F)$ such that $[I_a \alpha]_{j} \leq 0$ for every $j \in \{1, \ldots n\}$ such that $y_j \not\in \mathcal{Y}$. We have
\[[I_a \alpha]_{j} = \left( \mathop{\sum_{i=1}^{r+d}}_{\rho'(i) = j} \alpha_i \right) - \alpha_k \leq 0 \; \Longleftrightarrow \; \alpha_k \geq \mathop{\sum_{i=1}^{r+d}}_{\rho'(i) = j} \alpha_i.\]
where $\rho(k) = j$. It follows that $\alpha$ satisfies Condition $3$ of Definition \ref{balancing}.
\end{proof}

\section{Expanded Absorbing Set $\mathcal{Y}$}
\label{appendixb}

When applying Corollary \ref{maincorollary}, the choice of absorbing complex set $\mathcal{Y} \subseteq \mathcal{C}$ is often nontrivial. In fact, it is sometimes possible to show an extinction event occurs for one choice of $\mathcal{Y}$ but not for another (see Example 3.7 of the companion paper \cite{J-A-C-B}).

In this Appendix, we will consider an algorithmic way for starting with a minimal absorbing complex set $\mathcal{Y} \subseteq \mathcal{C}$ and systematically making it larger. We motivate the procedure with the following example.

\begin{example}
Consider the following CRN:
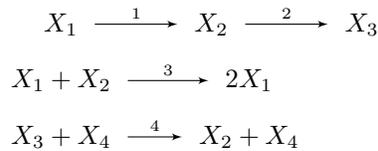
\begin{center}
\begin{tikzpicture}[auto, outer sep=3pt, node distance=2cm,>=latex']
\node (C1) {$X_1$};
\node [right of = C1, node distance = 2cm] (C2) {$X_2$};
\node [right of = C2, node distance = 2cm] (C3) {$X_3$};
\node [below of = C1, node distance = 0.75cm] (C4) {$X_1 + X_2$};
\node [right of = C4, node distance = 2.5cm] (C5) {$2X_1$};
\node [below of = C4, node distance = 0.75cm] (C6) {$X_3 + X_4$};
\node [right of = C6, node distance = 2.5cm] (C7) {$X_2 + X_4$};
\path[->] (C1) edge node [above = -0.15cm] {\tiny $1$} (C2);
\path[->] (C2) edge node [above = -0.15cm] {\tiny $2$} (C3);
\path[->] (C4) edge node [above = -0.15cm] {\tiny $3$} (C5);
\path[->] (C6) edge node [above = -0.15cm] {\tiny $4$} (C7);
\end{tikzpicture}
\end{center}
It can be seen directly that the CRN exhibits an extinction event since all of the $X_1$ can be converted into $X_2$ through reaction $1$. This shuts down reactions $1$ and $3$ and therefore eliminates the ability to replenish $X_1$.

We now attempt to obtain this result by applying Corollary \ref{maincorollary}. We need to first determine a $\mathcal{Y}$-admissible dom-CRN for some absorbing complex set $\mathcal{Y} \subseteq \mathcal{C}$ on the dom-CRN. We will attempt to determine the minimal such set $\mathcal{Y}$. This corresponds to the terminal complex set for the dom-CRN with the maximal set $\mathcal{D} \subseteq \mathcal{D}^*$. We have the following domination set:
\[\mathcal{D}^* = \{ X_1 + X_2 \to X_1, X_1 + X_2 \to X_2, 2X_1 \to X_1, X_2 + X_4 \to X_2, X_3 + X_4 \to X_3\}.\]
This produces the following dom-CRN:
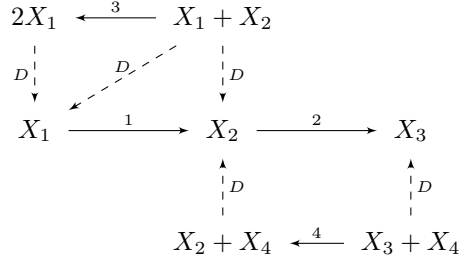
\begin{center}
\begin{tikzpicture}[auto, outer sep=3pt, node distance=2cm,>=latex']
\node (C1) {$X_1$};
\node [right of = C1, node distance = 2.5cm] (C2) {$X_2$};
\node [right of = C2, node distance = 2.5cm] (C3) {$X_3$};
\node [above of = C2, node distance = 1.5cm] (C4) {$X_1 + X_2$};
\node [left of = C4, node distance = 2.5cm] (C5) {$2X_1$};
\node [below of = C3, node distance = 1.5cm] (C6) {$X_3 + X_4$};
\node [left of = C6, node distance = 2.5cm] (C7) {$X_2 + X_4$};
\path[->] (C1) edge node [above = -0.15cm] {\tiny $1$} (C2);
\path[->] (C2) edge node [above = -0.15cm] {\tiny $2$} (C3);
\path[->] (C4) edge node [above = -0.15cm] {\tiny $3$} (C5);
\path[->] (C6) edge node [above = -0.15cm] {\tiny $4$} (C7);
\path[dashed,->] (C4) edge node [above = -0.15cm] {\tiny $D$} (C1);
\path[dashed,->] (C4) edge node [right = -0.15cm] {\tiny $D$} (C2);
\path[dashed,->] (C5) edge node [left = -0.15cm] {\tiny $D$} (C1);
\path[dashed,->] (C7) edge node [right = -0.15cm] {\tiny $D$} (C2);
\path[dashed,->] (C6) edge node [right = -0.15cm] {\tiny $D$} (C3);
\end{tikzpicture}
\end{center}
This dom-CRN is not admissible for $\mathcal{Y} = \{ X_3 \}$ since the domination reaction $X_3 + X_4 \to X_3$ leads to the terminal complex $X_3$. It can, however, be made admissible by removing this reaction. We therefore start the process by considering the dom-CRN with
\[
\mathcal{D} = \{ 2X_1 \stackrel{D_1}{\longrightarrow} X_1, X_1 + X_2 \stackrel{D_2}{\longrightarrow} X_1, X_1 + X_2 \stackrel{D_3}{\longrightarrow} X_2, X_2 + X_4 \stackrel{D_4}{\longrightarrow} X_2\}.
\]
The corresponding admissible dom-CRN is as follows, where $\mathcal{Y}$ is shaded blue, and a path common to all exterior forests is shown in bold red: 
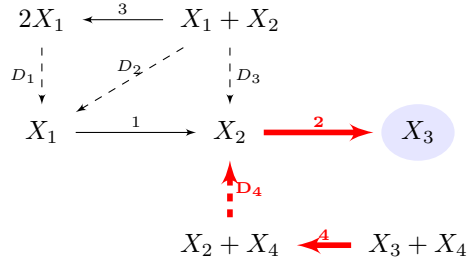
\begin{center}
\begin{tikzpicture}[auto, outer sep=3pt, node distance=2cm,>=latex']
\node (C1) {$X_1$};
\node [right of = C1, node distance = 2.5cm] (C2) {$X_2$};
\node [right of = C2, node distance = 2.5cm,ellipse,fill=blue!10] (C3) {$X_3$};
\node [above of = C2, node distance = 1.5cm] (C4) {$X_1 + X_2$};
\node [left of = C4, node distance = 2.5cm] (C5) {$2X_1$};
\node [below of = C3, node distance = 1.5cm] (C6) {$X_3 + X_4$};
\node [left of = C6, node distance = 2.5cm] (C7) {$X_2 + X_4$};
\path[->] (C1) edge node [above = -0.15cm] {\tiny $1$} (C2);
\path[red,line width=0.75mm,->] (C2) edge node [above = -0.15cm] {\tiny $\mathbf{2}$} (C3);
\path[->] (C4) edge node [above = -0.15cm] {\tiny $3$} (C5);
\path[red,line width=0.75mm,->] (C6) edge node [above = -0.15cm] {\tiny $\mathbf{4}$} (C7);
\path[dashed,->] (C4) edge node [above = -0.15cm] {\tiny $D_2$} (C1);
\path[dashed,->] (C4) edge node [right = -0.15cm] {\tiny $D_3$} (C2);
\path[dashed,->] (C5) edge node [left = -0.15cm] {\tiny $D_1$} (C1);
\path[red,line width=0.75mm,dashed,->] (C7) edge node [right = -0.15cm] {\tiny $\mathbf{D_4}$} (C2);
\end{tikzpicture}
\end{center}
There are several exterior forests which need to be checked. Since the path in bold red above is common to all exterior forests, however, they are all balanced by the following vector $\alpha = (\alpha_R,\alpha_D) \in \mathbb{R}_{\geq 0}^8$:
\[\begin{array}{r} \mbox{\emph{reaction:}} \\ \hline \alpha = \end{array}\hspace{-0.025in}\underbrace{\begin{array}{llll} \; 1 & 2 & 3 & 4 \\ \hline (0, & 1, & 0, & 1, \end{array} }_{\alpha_R}\hspace{-0.025in}\underbrace{\begin{array}{llll} D_1 & D_2 & D_3 & D_4 \\ \hline \; 0, & \; 0, & \; 0, & \; 1.\end{array} }_{\alpha_D}.\]
Consequently, we cannot conclude anything by Corollary \ref{maincorollary}.

It is suggestive, however, that the vector $\alpha$ which balances the CRN does not have support on the portion of the CRN which experiences an extinction event (specifically, the reactions involving $X_1$). This suggests that we redefine the set $\mathcal{Y}$ to correspond to the portion of the dom-CRN where the balancing vector $\alpha$ has support. We therefore expand $\mathcal{Y}$ to include the complexes in the red path above:
\[\mathcal{Y}' = \{ X_2, X_3, X_3+X_4, X_2+X_4 \} \supseteq \mathcal{Y}.\]
Notice that, given this absorbing complex set $\mathcal{Y}'$, we must modify the set $\mathcal{D}$ in order to allow the resulting dom-CRN to be admissible. In particular, we may not include $X_1 + X_2 \stackrel{D_3}{\longrightarrow} X_2$ or $X_2 + X_4 \stackrel{D_4}{\longrightarrow} X_2$. We have the reduced set
\[\mathcal{D}' = \{ 2X_1 \stackrel{D_1}{\longrightarrow} X_1, X_1 + X_2 \stackrel{D_2}{\longrightarrow} X_1 \} \subseteq \mathcal{D}.\]
The corresponding dom-CRN is as follows, where the complexes in $\mathcal{Y}'$ are shaded blue, and a $\mathcal{Y}'$-exterior forest is shown in red ($\mathcal{Y}'$-interior reactions omitted):
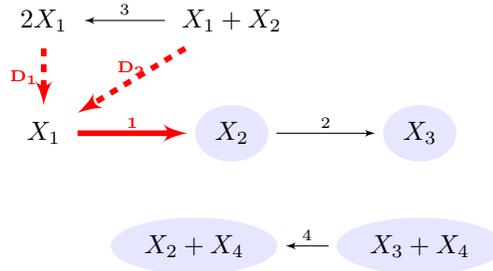
\begin{center}
\begin{tikzpicture}[auto, outer sep=3pt, node distance=2cm,>=latex']
\node (C1) {$X_1$};
\node [right of = C1, node distance = 2.5cm,ellipse,fill=blue!10] (C2) {$X_2$};
\node [right of = C2, node distance = 2.5cm,ellipse,fill=blue!10] (C3) {$X_3$};
\node [above of = C2, node distance = 1.5cm] (C4) {$X_1 + X_2$};
\node [left of = C4, node distance = 2.5cm] (C5) {$2X_1$};
\node [below of = C3, node distance = 1.5cm,ellipse,fill=blue!10] (C6) {$X_3 + X_4$};
\node [left of = C6, node distance = 3cm,ellipse,fill=blue!10] (C7) {$X_2 + X_4$};
\path[red,line width=0.75mm,->] (C1) edge node [above = -0.15cm] {\tiny $\mathbf{1}$} (C2);
\path[->] (C2) edge node [above = -0.15cm] {\tiny $2$} (C3);
\path[->] (C4) edge node [above = -0.15cm] {\tiny $3$} (C5);
\path[->] (C6) edge node [above = -0.15cm] {\tiny $4$} (C7);
\path[red,line width=0.75mm,dashed,->] (C4) edge node [above = -0.15cm] {\tiny $\mathbf{D_2}$} (C1);
\path[red,line width=0.75mm,dashed,->] (C5) edge node [left = -0.15cm] {\tiny $\mathbf{D_1}$} (C1);
\end{tikzpicture}
\end{center}
Notice that the set $\mathcal{Y}'$ is absorbing complex set but is not terminal, and that there is a unique path in the $\mathcal{Y}'$-exterior forest from each $\mathcal{Y}'$-exterior complex to the set $\mathcal{Y}'$. 
This $\mathcal{Y}'$-exterior forest is balanced if we can find a vector $\alpha = ((\alpha_R)_1, (\alpha_R)_2, (\alpha_R)_3, (\alpha_R)_4,(\alpha_D)_1, (\alpha_D)_2)$ such that $(\alpha_R)_1 > 0$ and:
\[\left\{ \; \; \begin{split} (\mbox{Cond.} \; 1): \; \; & \hspace{0.17in} (\alpha_R)_3 = 0 \\ (\mbox{Cond.} \; 2): \; \; & -(\alpha_R)_1 + (\alpha_R)_3 = 0 \\ & \hspace{0.17in} (\alpha_R)_1 - (\alpha_R)_2 - (\alpha_R)_3 + (\alpha_R)_4 = 0 \\ & \hspace{0.17in} (\alpha_R)_2 - (\alpha_R)_4 = 0 \\ (\mbox{Cond.} \; 3): \; \; & \hspace{0.17in} (\alpha_D)_1 + (\alpha_D)_2 \geq (\alpha_R)_1 \geq 0. \end{split} \right.\]
Condition 1 and the first equality in Condition 2 implies that $(\alpha_R)_1 = (\alpha_R)_3 = 0$ so that the forest in unbalanced. It follows from Corollary \ref{maincorollary} that the discrete state space system has a guaranteed extinction event and that the system is transient on $(\mathcal{Y}')^c = \{ 2X_1, X_1 + X_2, X_1\}$.
\end{example}

This example shows that it is sometimes possible to apply Corollary \ref{maincorollary} to guarantee an extinction event for one absorbing complex set $\mathcal{Y} \subseteq \mathcal{C}$ but not for another. While this is a limitation for our ability to apply Corollary \ref{maincorollary}, this example also suggests the following method for determining an initial set $\mathcal{Y}$ and then systematically expanding it:
\begin{enumerate}
\item
Initialize $\mathcal{Y}^* = \O$ and let $\mathcal{D}^*$ denote the domination set of the original CRN. 
\item
Set $\mathcal{Y} = \mathcal{Y}^*$ and $\mathcal{D} = \mathcal{D}^*$, then perform the following steps:
\begin{enumerate}
\item[(a)]
Add to $\mathcal{Y}$ all complexes which are connected by a directed path in the original CRN from a complex in $\mathcal{Y}$.
\item[(b)]
Add to $\mathcal{Y}$ any terminal complexes in the dom-CRN with $\mathcal{D}$.
\item[(c)]
Remove from $\mathcal{D}$ any domination reaction with its product complex in $\mathcal{Y}$.
\item[(d)]
Repeat (b) and (c) until the resulting dom-CRN is admissible.
\end{enumerate}
\item
For the $\mathcal{Y}$-exterior forests on this dom-CRN, attempt to compute balancing vectors $\alpha = (\alpha_R, \alpha_D) \in \mathbb{Z}_{\geq 0}^{r+d}$ with the maximal support (Definition \ref{balancing}).
\item
For a representative vector $\alpha$ from this set, let $\mathcal{Y}^*$ correspond to all the complexes corresponding to either the source or product complex of a reaction with $\alpha_k > 0$.
\item
Repeat steps 2-4 until either an unbalanced $\mathcal{Y}$-exterior forest is found, or $\mathcal{Y} = \mathcal{C}$. 
\end{enumerate}

\noindent Note that Condition 2(a) is only required if $\mathcal{Y}$ initially is not an absorbing complex set.

It is natural to wonder whether it is necessary to the continually refine the sets $\mathcal{Y}$ and $\mathcal{D}$ to ensure that the resulting dom-CRN is admissible. The following example shows that an incorrect conclusion may be reached if admissibility is not taken into account.

\begin{example}
Consider the following CRN, which corresponds to Example 3.10 of the companion paper \cite{J-A-C-B}:
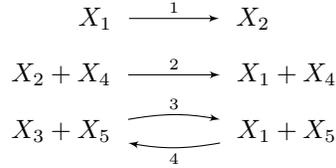
\begin{center}
\begin{tikzpicture}[auto, outer sep=3pt, node distance=2cm,>=latex']
\node (C1) {\; \; \; \; $X_1$};
\node [right of = C1, node distance = 3cm] (C2) {$X_2$ \; \; \; \;};
\node [below of = C1, node distance = 0.75cm] (C3) {$X_2 + X_4$};
\node [right of = C3, node distance = 3cm] (C4) {$X_1+X_4$};
\node [below of = C3, node distance = 0.75cm] (C5) {$X_3 + X_5$};
\node [right of = C5, node distance = 3 cm] (C6) {$X_1 + X_5$};
\path[->] (C1) edge node [above = -0.15cm] {\tiny $1$} (C2);
\path[->] (C3) edge node [above = -0.15cm] {\tiny $2$} (C4);
\path[->, bend left = 10] (C5) edge node [above = -0.15cm] {\tiny $3$} (C6);
\path[->, bend left = 10] (C6) edge node [below = -0.15cm] {\tiny $4$} (C5);
\end{tikzpicture}
\end{center}
We attempt to apply the steps contained above to seed the sets $\mathcal{Y}$ and $\mathcal{D}$ and then systematically refine them.

We start with the set $\mathcal{D}^* = \{ X_1 + X_4 \to X_1, X_1 + X_5 \to X_1, X_2 + X_4 \to X_2\}$. We set $\mathcal{Y}$ to correspond to the terminal complexes in the resulting dom-CRN, which gives $\mathcal{Y} = \{ X_2 \}$. Since the domination reaction $X_2 + X_4 \to X_2$ leads to $\mathcal{Y}$, we must remove it from $\mathcal{D}$ in order for the resulting dom-CRN to be admissible. The resulting dom-CRN is given by the following, where the only $\mathcal{Y}$-exterior forest is indicated in red, and the set $\mathcal{Y}$ is indicated in blue:
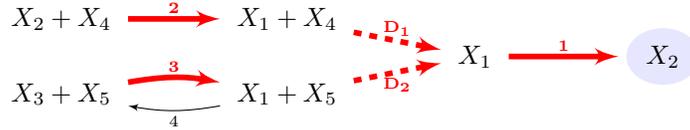
\begin{center}
\begin{tikzpicture}[auto, outer sep=3pt, node distance=2cm,>=latex']
\node (C3) {$X_2 + X_4$};
\node [right of = C3, node distance = 3cm] (C4) {$X_1+X_4$};
\node [below of = C4, node distance = 0.5cm] (ghost) {};
\node [below of = C3, node distance = 1cm] (C5) {$X_3 + X_5$};
\node [right of = C5, node distance = 3cm] (C6) {$X_1 + X_5$};
\node [right of = ghost, node distance = 2.5cm] (C1) {$X_1$};
\node [right of = C1, node distance = 2.5cm,ellipse,fill=blue!10] (C2) {$X_2$};
\path[red,line width=0.75mm,->] (C1) edge node [above = -0.15cm] {\tiny $\mathbf{1}$} (C2);
\path[red,line width=0.75mm,->] (C3) edge node [above = -0.15cm] {\tiny $\mathbf{2}$} (C4);
\path[red,line width=0.75mm,->, bend left = 10] (C5) edge node [above = -0.15cm] {\tiny $\mathbf{3}$} (C6);
\path[->, bend left = 10] (C6) edge node [below = -0.15cm] {\tiny $4$} (C5);
\path[red,line width=0.75mm,dashed,->] (C4) edge node [above = -0.15cm] {\tiny $\mathbf{D_1}$} (C1);
\path[red,line width=0.75mm,dashed,->] (C6) edge node [below = -0.15cm] {\tiny $\mathbf{D_2}$} (C1);
\end{tikzpicture}
\end{center}
The $\mathcal{Y}$-exterior forest indicated in red can be balanced by $\alpha = ((\alpha_R)_1,(\alpha_R)_2,(\alpha_R)_3,(\alpha_R)_4,$\\$(\alpha_D)_1,(\alpha_D)_2) = (1,1,0,0,1,0)$. While we cannot apply Corollary \ref{maincorollary} directly, we notice that $\alpha$ does not have full support on the reactions in the $\mathcal{Y}$-exterior forest above. We may therefore continue the algorithm presented earlier in this Appendix. We start by including in $\mathcal{Y}$ all those complexes which are contained in the reactions with $\alpha_k > 0$. This gives $\mathcal{Y} = \{ X_2 + X_4, X_1 + X_4, X_1, X_2\}$.

Now suppose we omit steps $2(a-d)$ in the algorithm presented earlier in this section. The resulting dom-CRN would be the following, where the only $\mathcal{Y}$-exterior forest is indicated in red ($\mathcal{Y}$-interior reactions not shown), and the absorbing complex set $\mathcal{Y}$ is shaded blue:
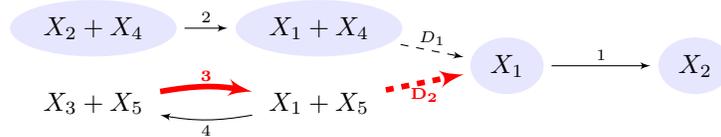
\begin{center}
\begin{tikzpicture}[auto, outer sep=3pt, node distance=2cm,>=latex']
\node [ellipse,fill=blue!10] (C3) {$X_2 + X_4$};
\node [right of = C3, node distance = 3cm,ellipse,fill=blue!10] (C4) {$X_1+X_4$};
\node [below of = C4, node distance = 0.5cm] (ghost) {};
\node [below of = C3, node distance = 1cm] (C5) {$X_3 + X_5$};
\node [right of = C5, node distance = 3cm] (C6) {$X_1 + X_5$};
\node [right of = ghost, node distance = 2.5cm,ellipse,fill=blue!10] (C1) {$X_1$};
\node [right of = C1, node distance = 2.5cm,ellipse,fill=blue!10] (C2) {$X_2$};
\path[->] (C1) edge node [above = -0.15cm] {\tiny $1$} (C2);
\path[->] (C3) edge node [above = -0.15cm] {\tiny $2$} (C4);
\path[red,line width=0.75mm,->, bend left = 10] (C5) edge node [above = -0.15cm] {\tiny $\mathbf{3}$} (C6);
\path[->, bend left = 10] (C6) edge node [below = -0.15cm] {\tiny $4$} (C5);
\path[dashed,->] (C4) edge node [above = -0.15cm] {\tiny $D_1$} (C1);
\path[red,line width=0.75mm,dashed,->] (C6) edge node [below = -0.15cm] {\tiny $\mathbf{D_2}$} (C1);
\end{tikzpicture}
\end{center}
This $\mathcal{Y}$-exterior forest cannot be balanced, which seems to suggest that the discrete state space system exhibits an extinction event by Corollary \ref{maincorollary}. We can, however, see that this dom-CRN is not admissible. To apply steps 2(a-d) of the algorithm presented earlier, we omit all domination reactions so that $\mathcal{D} = \O$ (step (d)). We must then add $X_1 + X_5$ and $X_3 + X_5$ to $\mathcal{Y}$ since these complexes are now terminal in the resulting dom-CRN. Since we then have $\mathcal{Y} = \mathcal{C}$ (step (e)), we terminate the procedure.

In fact, this example does not exhibit an extinction event for most initial conditions. Provided $\mathbf{X}_4 > 0$, $\mathbf{X}_5 > 0$, and any one of $\mathbf{X}_1$, $\mathbf{X}_2$, and $\mathbf{X}_3$ is positive, every complex is recurrent. This example therefore highlights the importance in guaranteeing that the resulting dom-CRN is admissible when expanding the absorbing complex set $\mathcal{Y}$. Care must be taken when expanding the absorbing complex set $\mathcal{Y}$. An example which is structurally identical, but which has a guaranteed extinction event, is given by Example 3.9 in the companion paper \cite{J-A-C-B}.
\end{example}

\section{Description of Algorithm \ref{alg:implementation} modules}
\label{appendixc}

In this Appendix, we present the details of the modules which are used in Algorithm \ref{alg:implementation}. Several of the modules involve mixed-integer linear programming (MILP). Recall that a MILP problem in the vector of decision variables $\mathbf{x} = (x_1, \ldots, x_k) \in \mathbb{R}^k$ may be written in the form
\begin{equation}
\label{milp}
\begin{split}
& \mbox{minimize} \; \; \; \mathbf{c}^T \mathbf{x} \\
\mbox{subject to} \; \; & \left\{ \begin{split} & A \; \mathbf{x} = \mathbf{a} \\
& B \; \mathbf{x} \leq \mathbf{b} \\
& x_j \in \mathbb{Z} \mbox{ for } j \in I, I \subseteq \{1, \ldots, k\} \end{split} \right.
\end{split}
\end{equation}
where $\mathbf{c} \in \mathbb{R}^k$, $A \in \mathbb{R}^{p \times k}$, $B \in \mathbb{R}^{q \times k}$, $\mathbf{a} \in \mathbb{R}^{p}$, and $\mathbf{b} \in \mathbb{R}^{q}$ \cite{Sz2}.

\begin{itemize}
\item[] \hspace{-0.4in} CreateModel($N$):\\This module takes the CRN $N$ and generates the structural matrices $Y$, $I_a$, $\Gamma$, $I_s$, and $A$. We utilize the CRN module package created for Python by Elisa Tonello which reads in a string file of reactions, and can also interface with SBML files \cite{tonello2016}.
\item[] \hspace{-0.4in} IsSubconservative($N$):\\This module determines whether a given CRN $N$ is subconservative or not by running the following MILP:
\begin{equation}
\label{conservativeLP}
\begin{split}
& \mbox{minimize} \; -z \\
\mbox{subject to} \; \; & \left\{ \begin{split} z - c_i & \leq 0, \mbox{ for } i \in \{1, \ldots, m\}\\
\mathbf{c}^T \Gamma & \leq \mathbf{0}^T 
 \end{split} \right.
\end{split}
\end{equation}
over the vector of decision variables $\mathbf{c} = (c_1, \ldots, c_m) \in \mathbb{R}_{\geq 0}^m$. The first constraint set guarantees all components of $\mathbf{c}$ are strictly positive if they can be, while the second constraint set imposes that the CRN is subconservative. If (\ref{conservativeLP}) returns an optimal value less than zero, the module returns {\tt true}; otherwise, it returns {\tt false}. If the CRN is subconservative, it furthermore runs the module IsConservative($N$), which replaces the second constraint set in (\ref{conservativeLP}) with $\mathbf{c}^T \Gamma = \mathbf{0}^T$.

\item[] \hspace{-0.4in} DominationSet($N$):\\This modules determines the domination set $\mathcal{D}^*$ of the CRN (Definition \ref{complexdomination}) by looping over all pairs of complexes $y_i$ and $y_j$, $i \not= j$, and assigning $D^*_{ij} = 1$ if $y_j \leq y_i$ and $D^*_{ij} = 0$ otherwise.
\item[] \hspace{-0.4in} FindTerm($N$): This module takes in the Laplacian $A$ of the CRN $N$ and runs the following linear program:
\begin{equation}
\label{findterm}
\begin{split}
& \mbox{minimize} \; -\sum_{i=1}^n y_i \\
& \mbox{subject to} \; \;
 \begin{split} \; \; \displaystyle{A \; \mathbf{y}} & = 0 \end{split}
\end{split}
\end{equation}
over $\mathbf{y} \in [0,1/\epsilon]^n$. By Theorem \ref{kerneltheorem}, this has a solution only on the support of the terminal complexes in the CRN, and by Lemma 3.1 of \cite{J-A-C-B} this contains the terminal complexes of every dom-CRN. If $y_i > 0$,  we set $Y_i = 1$; otherwise $Y_i = 0$.
\item[] \hspace{-0.4in} FindAdmissibleDom($N$,$\mathcal{Y}^*$,$\mathcal{D}^*$):\\Given the domination set $\mathcal{D}^*$ and an initial complex set $\mathcal{Y}^* \subseteq \mathcal{C}$, this module determines the minimal absorbing complex set $\mathcal{Y} \supseteq \mathcal{Y}^*$ and corresponding maximal set $\mathcal{D} \subseteq \mathcal{D}^*$ such that the resulting dom-CRN is $\mathcal{Y}$-admissible (Definition \ref{dominationnetwork}). We accomplish this by running a loop which: (a) Adds complexes to $\mathcal{Y}$ which are connected by a path in the original CRN from a complex in $\mathcal{Y}$; (b) Adds complexes to $\mathcal{Y}$ which are terminal in the dom-CRN for the current $\mathcal{D}$; and (c) removes reactions in $\mathcal{D}$ which lead to complexes in the current set $\mathcal{Y}$. The loop is run until the set $\mathcal{Y}$ does not change from one iteration to the next.

\item[] \hspace{-0.4in} DominationExpandedNetwork($N$, $D$):\\This module generates structural matrices $\Gamma$, $Y$, $I_a$, $I_s$, and $A$ for the admissible dom-CRN found in the previous module (Definition \ref{dominationnetwork}). The module indexes the first $r$ reactions as in the original CRN and then adds the domination reactions with indices $\{r+1,\ldots,r+d\}$ where $d = |\mathcal{D}|$.
\item[] \hspace{-0.4in} CycleForests(dom-$N$):\\This module cycles over every combination of reactions $F$ such that every $\mathcal{Y}$-exterior reaction in the dom-CRN is the source for exactly one reaction. The exterior forest conditions of Theorem \ref{foresttheorem} are checked by the module  IsExtForest($F$). If this module returns {\tt true}, the module IsBalanced($F$) is run to determine if the forest is balanced in accordance with Definition \ref{balancing} by Theorem \ref{tieringtheorem}. If the module finds an exterior forest which is not balanced then the module confirms that the discrete state space CRN exhibits a extinction event in accordance with Corollary \ref{maincorollary}. If $F$ is balanced, it runs ExpandedY(F); otherwise, the module returns {\tt false}.
\item[] \hspace{-0.4in} IsExtForest($F$):\\This module checks whether $F$ is an exterior forest (Definition \ref{forest1}) according to Theorem \ref{foresttheorem}. The module takes in the adjacency matrix $I_a$ for the dom-CRN and $F$. We let $F_i = 1$ if $R_i \in \mathcal{R}_F \cup \mathcal{D}_F$ and $F_i = 0$ otherwise, and set a small parameter $\epsilon > 0$. We then run the following linear program:
\begin{equation}
\label{checkforest}
\begin{split}
& \mbox{minimize} \; -\sum_{i=1}^{r+d} v_i \\
\mbox{subject to} \; \;
& \left\{ \begin{split} \epsilon \; F_i - v_i & \leq 0, \; \mbox{ for } i \in \{1, \ldots, r+d\}\\
v_i - \frac{1}{\epsilon} \; F_i & \leq 0, \; \mbox{ for } i \in \{1, \ldots, r+d\} \\
\sum_{j=1}^{r+d} [I_a]_{ij} \; v_j & \leq -\epsilon + \frac{1}{\epsilon} \; Y_i, \; \mbox{ for } i \in \{ 1, \ldots, n\}\end{split} \right.
\end{split}
\end{equation}
over $\mathbf{v} = (v_1, \ldots, v_n) \in \mathbb{R}_{\geq 0}^n$. The first two constraint sets in (\ref{checkforest}) ensure that $\mathbf{v}$ has the same support on $\mathcal{R}_F \cup \mathcal{D}_F$. The third constraint set guarantees that, if $y_i \not\in \mathcal{Y}$ (i.e. $Y_i = 0$) then $[I_d \; \mathbf{v}]_i \leq -\epsilon < 0$ (there is no restriction if $y_i \in \mathcal{Y}$, i.e. $Y_i = 1$). If (\ref{checkforest}) is feasible, the module returns the value {\tt true}; otherwise, it returns {\tt false}.
\item[] \hspace{-0.4in} IsBalanced($F$):\\This module checks whether $F$ is balanced (Definition \ref{balancing}) according to Theorem \ref{tieringtheorem}. It takes in the stoichiometric matrix $\Gamma$ of the CRN, the adjacency matrix $I_a$ of the dom-CRN, and the exterior forest $F$. We let $F_i = 1$ if $R_i \in \mathcal{R}_F \cup \mathcal{D}_F$ and $F_i = 0$ otherwise, and set a small parameter $\epsilon > 0$. We then run the following linear program:
\begin{equation}
\label{checkbalancing}
\begin{split}
& \mbox{minimize} \; -\sum_{i=1}^{r} \alpha_i \\
\mbox{subject to} \; \;
& \left\{ \begin{split}\alpha_i - \frac{1}{\epsilon} \; F_i & \leq 0, \; \mbox{ for } i \in \{1, \ldots, r+d\} \\
\sum_{j=1}^{r} \Gamma_{ij} \; \alpha_j & = 0, \; \mbox{ for } i \in \{1, \ldots, m\} \\
\sum_{j=1}^{r+d} [I_a]_{ij} \; v_j & \leq \frac{1}{\epsilon} \; Y_i, \; \mbox{ for } i \in \{1, \ldots, n\}\end{split} \right.
\end{split}
\end{equation}
over $\alpha = (\alpha_1, \ldots, \alpha_{r+d}) \in \mathbb{R}_{\geq 0}^{r + d}$. The first constraint guarantees that $R_i \not\in \mathcal{R}_F \cup \mathcal{D}_F$ (i.e. $F_i = 0$) then $\alpha_i = 0$. The second constraint set guarantees that $\Gamma \; \alpha_R = 0$. The third constraint set guarantees that, if $y_i \not\in \mathcal{Y}$ (i.e. $Y_i = 0$) then $[I_a^d \; \alpha]_i \leq 0$ (there is no restriction if $y_i \in \mathcal{Y}$, i.e. $Y_i = 1$). If (\ref{checkbalancing}) the optimal value zero, corresponding to the trivial vector $\alpha_R = \mathbf{0}$, the module returns {\tt false}; otherwise, it returns {\tt true}. It also returns the vector $\alpha$.
\item[] \hspace{-0.4in} ExpandedY($F$):\\For a given vector $\alpha = (\alpha_R,\alpha_D) \in \mathbb{Z}_{\geq 0}^{r+d}$, this module returns the set of complexes in the reactions $R_k = (y_i, y_j) \in \mathcal{R}_F \cup \mathcal{D}_F$ such that $\alpha_k > 0$. Note that if $\alpha$ has full support, then we set $\mathcal{Y} = \mathcal{C}$, which is one of the stopping criteria for Algorithm \ref{alg:implementation}.
\item[] \hspace{-0.4in} WriteOutput:\\This module writes the output into a {\tt .dat} file.
\end{itemize}

\section{Result of BioModels Database Search}
\label{appendixd}

The following 86 models from the European Bioinformatics BioModels Database were identified by the algorithm as being subconservative and having a discrete extinction event. The models are labeled with (C) if they were conservative as well as subconservative, (S) is they had a source only species, and (P) if they had a product only species.\\

\scriptsize
\begin{multicols}{3}
\begin{enumerate}
\item[]
{\tt biomd0000000013(C)(S)(P)}
\item[]
{\tt biomd0000000017(C)(S)(P)}
\item[]
{\tt biomd0000000040(C)(P)}
\item[]
{\tt biomd0000000046(C)(S)(P)}
\item[]
{\tt biomd0000000050(S)(P)}
\item[]
{\tt biomd0000000052(S)(P)}
\item[]
{\tt biomd0000000069(C)(S)}
\item[]
{\tt biomd0000000071(S)(P)}
\item[]
{\tt biomd0000000080(C)(S)}
\item[]
{\tt biomd0000000082(C)(S)}
\item[]
{\tt biomd0000000090(S)(P)}
\item[]
{\tt biomd0000000104(C)(S)(P)}
\item[]
{\tt biomd0000000172(C)(P)}
\item[]
{\tt biomd0000000176(C)(P)}
\item[]
{\tt biomd0000000177(C)(S)(P)}
\item[]
{\tt biomd0000000178(C)(S)(P)}
\item[]
{\tt biomd0000000190(S)}
\item[]
{\tt biomd0000000191(S)}
\item[]
{\tt biomd0000000198(C)(S)}
\item[]
{\tt biomd0000000199(C)(S)(P)}
\item[]
{\tt biomd0000000203(S)(P)}
\item[]
{\tt biomd0000000204(S)(P)}
\item[]
{\tt biomd0000000209(C)(S)(P)}
\item[]
{\tt biomd0000000210(C)(S)(P)}
\item[]
{\tt biomd0000000211(C)(P)}
\item[]
{\tt biomd0000000225(C)(S)(P)}
\item[]
{\tt biomd0000000233(C)(S)(P)}
\item[]
{\tt biomd0000000236(S)}
\item[]
{\tt biomd0000000243(S)(P)}
\item[]
{\tt biomd0000000245(S)(P)}
\item[]
{\tt biomd0000000253(S)}
\item[]
{\tt biomd0000000262(S)}
\item[]
{\tt biomd0000000263(S)}
\item[]
{\tt biomd0000000264(S)}
\item[]
{\tt biomd0000000267(C)(S)(P)}
\item[]
{\tt biomd0000000281(S)}
\item[]
{\tt biomd0000000282(C)(S)(P)}
\item[]
{\tt biomd0000000283(C)(S)(P)}
\item[]
{\tt biomd0000000296(C)(S)(P)}
\item[]
{\tt biomd0000000305(C)(S)(P)}
\item[]
{\tt biomd0000000320(S)}
\item[]
{\tt biomd0000000321(S)}
\item[]
{\tt biomd0000000333(C)(S)(P)}
\item[]
{\tt biomd0000000335(C)(S)(P)}
\item[]
{\tt biomd0000000336(S)}
\item[]
{\tt biomd0000000357(C)(S)(P)}
\item[]
{\tt biomd0000000358(C)(S)(P)}
\item[]
{\tt biomd0000000359(C)(S)}
\item[]
{\tt biomd0000000360(C)(S)}
\item[]
{\tt biomd0000000361(C)(S)}
\item[]
{\tt biomd0000000362(C)(S)(P)}
\item[]
{\tt biomd0000000363(C)(S)(P)}
\item[]
{\tt biomd0000000364(S)(P)}
\item[]
{\tt biomd0000000366(C)(S)(P)}
\item[]
{\tt biomd0000000383(S)}
\item[]
{\tt biomd0000000384(S)}
\item[]
{\tt biomd0000000385(S)}
\item[]
{\tt biomd0000000386(S)}
\item[]
{\tt biomd0000000387(S)}
\item[]
{\tt biomd0000000388(S)(P)}
\item[]
{\tt biomd0000000389(C)(S)(P)}
\item[]
{\tt biomd0000000390(C)(S)}
\item[]
{\tt biomd0000000411(S)(P)}
\item[]
{\tt biomd0000000415(S)(P)}
\item[]
{\tt biomd0000000436(S)}
\item[]
{\tt biomd0000000437(S)(P)}
\item[]
{\tt biomd0000000438(C)(S)(P)}
\item[]
{\tt biomd0000000464(C)(S)(P)}
\item[]
{\tt biomd0000000465(C)(S)(P)}
\item[]
{\tt biomd0000000475(C)(S)(P)}
\item[]
{\tt biomd0000000478(C)(S)(P)}
\item[]
{\tt biomd0000000481(C)(S)(P)}
\item[]
{\tt biomd0000000495(S)}
\item[]
{\tt biomd0000000502(S)}
\item[]
{\tt biomd0000000513(C)(P)}
\item[]
{\tt biomd0000000523(C)(S)(P)}
\item[]
{\tt biomd0000000524(C)(S)(P)}
\item[]
{\tt biomd0000000525(C)(S)(P)}
\item[]
{\tt biomd0000000526(C)(S)(P)}
\item[]
{\tt biomd0000000529(C)(S)(P)}
\item[]
{\tt biomd0000000540(C)(S)(P)}
\item[]
{\tt biomd0000000541(C)(S)(P)}
\item[]
{\tt biomd0000000546(C)(S)(P)}
\item[]
{\tt biomd0000000566(C)(S)}
\item[]
{\tt biomd0000000567(C)(S)}
\item[]
{\tt biomd0000000572(S)(P)}
\end{enumerate}
\end{multicols}

\end{appendices}

\end{document}